\documentclass[11pt, oneside,reqno]{article}
\usepackage[top=2.54cm, bottom=3cm, right=3cm, left=3cm]{geometry}
\usepackage[utf8]{inputenc}
\usepackage[english]{babel}
\usepackage{amsmath, mathtools, amsthm, amssymb}
\usepackage{dsfont}
\usepackage{changepage}
\usepackage{caption, subcaption}
\usepackage{enumitem}
\usepackage{hyperref}
\usepackage{float}
\usepackage{titling}
\usepackage{cite}
\usepackage[titletoc,toc,title]{appendix}
\usepackage{color}
\DeclareMathAlphabet{\mathdutchcal}{U}{dutchcal}{m}{n}

\newcommand\comp{\mathrel{\overset{\makebox[0pt]{\mbox{\normalfont\tiny\sffamily cmp}}}{\sim}}}
\newcommand\ncomp{\mathrel{\overset{\makebox[0pt]{\mbox{\normalfont\tiny\sffamily cmp}}}{\nsim}}}

\theoremstyle{plain}
\newtheorem{thm}{Theorem}[section] % reset theorem numbering for each section

\newtheorem{corollary}{Corollary}[thm]
\newtheorem{lemma}[thm]{Lemma}
\newtheorem{prop}[thm]{Proposition}

\theoremstyle{definition}
\newtheorem{defn}[thm]{Definition} % definition numbers are dependent on theorem numbers
\newtheorem{exmp}[thm]{Example}

\makeatletter
\def\thmheadassumption#1#2#3{%
	\thmname{#1}\thmnumber{\@ifnotempty{#1}{ }\@upn{#2}}%
	\thmnote{ {\the\thm@notefont[#3]}}}
\makeatother

\newtheoremstyle{assumption}% Name
{}% space above
{}% space below
{\itshape}% body font
{}% indent
{\bfseries}% head font
{.}% punctuation after head
{ }% space after head (has to be space or dimension!)
{\thmheadassumption{#1}{(A#2)}{#3}}% head spec

\theoremstyle{assumption}

\theoremstyle{remark}
\newtheorem{rmk}{Remark}
\newtheorem*{notation}{Notation}

\newcommand{\R}{\mathbb{R}}
\newcommand{\N}{\mathbb{N}}
\newcommand{\Z}{\mathbb{Z}}
\newcommand{\Prob}{\mathbb{P}}

\newcommand{\dt}{\, \mathrm{d}t}
\newcommand{\di}{\, \mathrm{d}}
\newcommand{\tvnorm}{\mbox{\normalfont\scriptsize\sffamily TV}}
\newcommand{\tvnormbig}{\mbox{\normalfont\sffamily TV}}

\newcommand{\norm}[1]{\left\lVert #1 \right\rVert}

\newcommand{\hilbert}{\mathcal H}

\DeclareMathOperator*{\esssup}{ess\,sup}
\DeclareMathOperator*{\essinf}{ess\,inf}

\makeatletter
\newcommand{\subalign}[1]{%
	\vcenter{%
		\Let@ \restore@math@cr \default@tag
		\baselineskip\fontdimen10 \scriptfont\tw@
		\advance\baselineskip\fontdimen12 \scriptfont\tw@
		\lineskip\thr@@\fontdimen8 \scriptfont\thr@@
		\lineskiplimit\lineskip
		\ialign{\hfil$\m@th\scriptstyle##$&$\m@th\scriptstyle{}##$\hfil\crcr
			#1\crcr
		}%
	}%
}
\makeatother

\begin{document}
	
	\title{Hyperbolic contractivity and the Hilbert metric on probability measures}
	
	\author{Samuel N. Cohen\thanks{Mathematical Institute, University of Oxford and Alan Turing Institute, {\tt cohens@maths.ox.ac.uk}.}
		\and
		Eliana Fausti\thanks{Department of Mathematics, Imperial College London, {\tt eliana.fausti@imperial.ac.uk}.}}
	
	\date{\today}
	
	\maketitle
	
	\begin{abstract}
		This paper gives a self-contained introduction to the Hilbert projective metric $\hilbert$ and its fundamental properties, with a particular focus on the space of probability measures. We start by defining the Hilbert pseudo-metric on convex cones, focusing mainly on dual formulations of $\hilbert$. We show that linear operators on convex cones contract in the distance given by the hyperbolic tangent of $\hilbert$, which in particular implies Birkhoff's classical contraction result for $\hilbert$. Turning to spaces of probability measures, where $\hilbert$ is a metric, we analyse the dual formulation of $\hilbert$ in the general setting, and explore the geometry of the probability simplex under $\hilbert$ in the special case of discrete probability measures. Throughout, we compare $\hilbert$ with other distances between probability measures. In particular, we show how convergence in $\hilbert$ implies convergence in total variation, $p$-Wasserstein distance, and any $f$-divergence. Finally, we derive a novel sharp bound for the total variation between two probability measures in terms of their Hilbert distance. 
	\end{abstract}
	
	\begin{adjustwidth}{0.6cm}{0.6cm}
		\begin{flushleft}
			
		\end{flushleft}
	\end{adjustwidth}
	
	\section{Introduction}
	
	The main goal of this paper is to provide a clear, self-contained guide to the Hilbert projective (pseudo-)metric and its merits in applied probability theory. In particular, as we discuss in more detail below, this metric may serve as a powerful tool for understanding whether linear operators on probability measures contract, in what sense, and at what rate. This, of course, is part of a much bigger story concerning the contraction of linear operators on positive cones---of which the space of (non-negative) measures is a prominent example.
	
	In \cite{birkhoff57}, Garrett Birkhoff\footnote{The son of George D.~Birkhoff, who is better known in probabilistic circles for his proof of the ergodic theorem.} proved that positive linear operators on a positive convex cone contract in the Hilbert projective (pseudo-)metric (introduced by Hilbert in \cite{hilbert1895}). Starting from this result, he easily derives a proof of the famous Perron theorem, by essentially reducing it to a special case of the Banach fixed point theorem. Similarly, he also immediately proves several of its generalizations: first, the extension of the Perron theorem to non-negative matrices due to Frobenius, then the extensions to infinite dimensional function spaces and positive integral and compact operators, originally due to Jentzsch \cite{jentzsch} and Krein and Rutman \cite{krein_rutman48}. In the same year as Birkhoff, Samelson \cite{samelson57} also published a proof of the Perron--Frobenius theorem using projective geometry, and similarly a few years later Hopf \cite{hopf63} presented a contraction result for positive linear integral operators and an alternative proof of Jentzsch's theorem (apparently without being aware of Birkhoff's previous results). The combination of the Hilbert projective metric with the contraction mapping theorem has inspired a vast amount of further work on the Perron--Frobenius theorem and its various extensions, see  Lemmens and Nussbaum \cite{nonlinear_perron12} and references therein for a detailed overview of the topic. We highlight the early works of Thompson \cite{thompson63}, Bushell \cite{bushell73}, and Kohlberg and Pratt \cite{kohlberg82}, as well as the recent works by Rugh \cite{rugh10} and Dubois \cite{dubois09} on operators on complex cones which have strongly influenced the presentation in this paper.
	
	While the impact on Perron--Frobenius theory is certainly the most significant consequence of Birkhoff's work, we became interested in the Hilbert metric and Birkhoff's contraction result due to a different (although related) application: the study of the ergodicity of non-negative linear operators. Work in this direction was presented by Birkhoff himself in \cite[Chapter~XVI,~Sec.7-8]{birkhoff_lattice}, Seneta and Sheridan \cite{seneta81, seneta_matrices} and Le Gland and Mevel \cite{legland00_matrices}. Clearly, the ergodic theory of Markov processes with linear transition kernels can be reduced to a particular case of this broader topic. Nevertheless, the Hilbert metric is not strictly necessary for developing the standard ergodic theory of Markov processes. For example, in \cite[Chapter~3]{seneta_matrices} Seneta uses the Hilbert metric to study the ergodicity of inhomogenous products of non-negative matrices; on the other hand, in the following chapter \cite[Chapter~4]{seneta_matrices}, which treats specifically the ergodicity of discrete Markov chains, the Hilbert metric is not used, since the analysis simplifies by virtue of the generators being stochastic matrices.
	
	Thus, one of the main avenues of application for the Hilbert projective metric in probability theory is, in fact, relatively unknown, which might in part explain why, despite being quite powerful tools, neither the Hilbert metric nor Birkhoff's contraction results seem to have found widespread use in the probability community\footnote{To illustrate this point, no mention is made of this approach in any of the books by Ethier and Kurtz \cite{ethier2009markov}, Meyn and Tweedie \cite{meyn2009markov}, Br\'emaud \cite{bremaud2017discrete}, Kallenberg \cite{kallenberg1997foundations}, or Jacod and Shiryaev \cite{jacod2003limit}.}. There are of course a few notable exceptions: Birkhoff's contraction theorem was employed, for example, in an elegant proof of the linear convergence of the Sinkhorn algorithm by Franklin and Lorenz \cite{franklin_lorenz_sinkhorn} (see also a recent extension of this result by Eckstein \cite{eckstein24}). It also proved fundamental in the qualitative and quantitative analysis of the stability of hidden Markov processes, where the use of the Hilbert metric was introduced by Atar and Zeitouni \cite{atar-zeitouni-lyapunov97,atar-zeitouni97}, and, following their work, became a well-established approach to the problem, see e.g.~\cite{legland00_hmm,legland04,bax-chiga-lip04,cohen_fausti_23}. More recently, the Hilbert metric has found computational applications in entropic interpolation \cite{entropic_disp_hilbert} and nonlinear embeddings \cite{nielsen22}.
	
	In writing this paper, we had in mind applied probabilists and statisticians who find themselves curious, as we were, about the merits and limitations of working with the Hilbert projective metric in a probabilistic context. Consequently, our main focus is the study of the space of probability measures when equipped with the Hilbert projective metric.
	
	Nevertheless, we shall start in greater generality, by defining the Hilbert projective (pseudo-)metric on a proper convex cone in a locally convex topological vector space. From there, our main contribution in Section~\ref{sec:defn_and_contraction} is the introduction of a new (pseudo-)metric, the hyperbolic tangent of the Hilbert metric (which we call the $\mathcal{T}$-distance), under which we prove linear operators also contract (see Definition~\ref{def:t-distance} and Theorem~\ref{thm:tanh_contaction}). The advantage of using the $\mathcal{T}$-distance compared to the Hilbert projective metric is that $\mathcal{T}$ stays bounded, while the Hilbert metric (easily) diverges to infinity. As far as we are aware, the formulation of this contraction result in Theorem~\ref{thm:tanh_contaction} has not been given before (but we note that our proof is inspired by Dubois' proof of \cite[Thm.~2.3]{dubois09}, so we do not wish to claim full credit for the result).
	
	When introducing the Hilbert projective metric, we especially insist on its definition through duality, which we first came across in \cite{rugh10}. This `dual' definition, in particular through a predual space, is natural in the context of probability measures, where distances are often defined by considering measures as integrators dual to particular classes of functions. In Section~\ref{sec:hilbert_metric_measures} we provide a careful analysis of the Hilbert metric and the topology it induces on the space of probability measures using duality, and we show that convergence of measures in the Hilbert metric (or in the $\mathcal{T}$-distance) is stronger than convergence in total variation (as was already shown in \cite[Lemma~1]{atar-zeitouni97}), convergence in $p$-Wasserstein distance and convergence in any $f$-divergence.
	
	In Section~\ref{sec:geometry} we study the geometry of the probability simplex under the Hilbert metric. Using our dual approach, we give an easy derivation of the explicit formula for the contraction rate of a linear operator (Proposition~\ref{prop:contraction_coeff}), which can also be extended to infinite settings (Proposition~\ref{prop:bikhoff_coeff_kernel}). As the probability simplex is finite dimensional, it has a natural manifold structure, however the Hilbert metric does not induce a hyperbolic (in the sense of Gromov) geometry on the simplex (since its boundary is not differentiable, see \cite{Benoist03}), nor a Riemannian structure. The Hilbert geometry is far more curious: we find an explicit characterizations of Hilbert balls as `hexagonal' convex polytopes, extending to the $n$-dimensional case work by Phadke \cite{phadke_hexagons75} and de la Harpe \cite{delaHarpe_simplices}. Finally, in the last section of the paper we use these geometric observations to prove a sharp bound for the total variation distance between two probability measures in terms of their $\mathcal{T}$-distance (see Theorem~\ref{thm:hilbert_and_L1_simplex} and Cor.~\ref{cor:tv_H_bound_prob_measures}).
	
	\section{The Hilbert projective metric: definitions and contractivity}\label{sec:defn_and_contraction}
	
	Let us start with the definition of the Hilbert projective distance, in the sense of Birkhoff \cite{birkhoff57,birkhoff_lattice}, on a cone in a (real) locally convex topological vector space (LCS). Note that Birkhoff works more specifically with cones in real Banach spaces (lattices). However, as we will see shortly, the definition of the metric does not require the space to be Banach, so for the sake of generality we work with an LCS.
	
	Let $X$ be an LCS. Let $C \subset X$ be a proper closed convex cone, meaning that $C$ is closed and satisfies
	\begin{equation*}
		C+C \subseteq C, \qquad \R_{+} C = C, \qquad C \cap -C = \{ 0 \}.
	\end{equation*}
	Following \cite[Sec.~4]{rugh10} (or extrapolating directly from \cite{birkhoff57} or \cite[Chapter~XVI]{birkhoff_lattice}), we give Birkhoff's definition of the Hilbert projective distance.
	\begin{defn}[Hilbert projective pseudo-metric]
		For $x, y \in C \setminus \{0\}$, where $C$ is a proper closed convex cone in a real LCS, let $\beta(x,y) \in (0, \infty]$ be given by
		\begin{equation*}
			\beta(x,y) = \inf \{ r > 0 \, : \, r x - y \in C \} = \sup \{ r > 0 \, : \, r x - y \notin C \}.
		\end{equation*}
		Then the Hilbert projective distance is defined by
		\begin{equation}\label{eq:hilbert_beta}
			\hilbert (x, y)  = \log \big( \beta(x,y) \beta (y,x) \big) \in [0, \infty], \qquad \forall x,y \in C \setminus \{0\}.
		\end{equation}
	\end{defn}
	It is worth spending a few moments to properly understand these definitions. Since $C$ is closed, and $-y \notin C$ for all $y \in C \setminus \{0\}$, we have $\beta(x,y)>0$. It is then straightforward to see that $\beta(x,y) \beta(y,x)\ge 1$, by noting that
	\begin{align}
		\frac{1}{\beta(x,y)}
		&= \inf \{ 1/r > 0 \, : \, r x - y \notin C \} 
		= \inf \{ r > 0 \, : \, x - ry \notin C \} 
		= \inf \{ r > 0 \, : \, ry -x \notin - C \}  \nonumber \\
		&\le \inf \{ r > 0 \, : \, ry -x \in C \} = \beta(y,x), \label{eq:ineq_beta_beta_inv}
	\end{align}
	and $\beta(x,y) = \frac{1}{\beta(y,x)}$ if and only if $y = cx$ for some $c \in \R_{+}$ (in this case, $x$ and $y$ are said to be \textit{collinear}). Hence \eqref{eq:hilbert_beta} is well-defined with $\hilbert (x, y)=0$ if and only if $x$ and $y$ are collinear. Symmetry of $\hilbert$ is clear from the definition, and one can verify that the triangle inequality is also satisfied (see Remark \ref{rmk:triangle}): then, $\hilbert$ is a pseudo-metric for $C$ (see also \cite[Chapter~XVI]{birkhoff_lattice}).
	
	\begin{exmp}
		The finite non-negative measures on $[0,1]$ form a proper convex (non-negative) cone in the space of the signed measures on $[0,1]$. Since the signed measures on $[0,1]$ equipped with the total variation distance form a Banach space (and therefore an LCS), the definition above applies. The Hilbert pseudo-metric can then be restricted to the probability measures on $[0,1]$, to give a true metric. We will explore in detail the Hilbert distance on the space of measures in Section~\ref{sec:hilbert_metric_measures}.
	\end{exmp}
	
	\begin{rmk}
		The only role of the specific choice of LCS topology on $X$, in defining the $\hilbert$-metric, is to have the notion of $C$ being closed in $X$. Two different topologies on $X$ for which $C$ is closed in $X$ will both give rise to the same Hilbert projective metric. In this sense, the Hilbert metric on $C$ is independent of the topology on the ambient space $X$.
	\end{rmk}
	
	By closure of $C$ in $X$, if $\beta(x,y) < \infty$, we must have that $\beta(x,y) x - y \in \partial C$, where $\partial C$ is the boundary of $C$. If $x \in \partial C$ and $y \in \mathring{C}$, where $\mathring{C}$ denotes the interior of $C$, then $\beta(x,y) = \infty$. However, if both $x,y \in \partial C$, then $\beta(x,y)$ might be finite. Using Birkhoff's choice of terminology \cite[Chapter XVI]{birkhoff_lattice}, $x,y \in C$ are \textit{comparable} if $\beta(x,y)$ and $\beta(y,x)$ are both finite. We observe that two boundary elements $x,y \in \partial C$ might still be comparable.
	
	Especially when $X$ is infinite dimensional, it can be useful to understand the Hilbert projective distance through duality (see \cite{rugh10,dubois09}, where this is exploited in the analysis of complex cones). Let $X^{\ast}$ be the (topological) dual of $X$, and let $\langle \cdot , \cdot  \rangle$ be the natural bilinear form $X^{\ast} \times X \to \R$. One can define the dual cone $C^{\ast}$ as
	\begin{equation}\label{eq:dual_cone}
		C^{\ast} = \big\{ f \in X^{\ast} \, : \, f |_{C} \ge 0   \big\}.
	\end{equation}
	\begin{prop}
		An equivalent definition of the Hilbert pseudo-metric is given by
		\begin{equation}\label{eq:hilbert_dual}
			\hilbert (x,y) = \sup_{\substack{f,g \in C^{\ast} \\\langle f,x \rangle, \langle g,y \rangle \neq 0} } \bigg\{ \log \frac{\langle f,y \rangle \langle g,x \rangle}{\langle f,x \rangle \langle g,y \rangle} \bigg\}.
		\end{equation}
	\end{prop}
	\begin{proof}
		Let
		$
		\tilde C := \big\{ x \in X \, : \, \langle f, x \rangle \ge 0, \, \, \forall f \in C^{\ast} \big\}.
		$
		One can confirm that $\tilde C$ is a proper closed convex cone. We clearly have that $C \subseteq \tilde C$. Now consider $x \notin C$. Since $C$ is convex and closed, by the Geometric Hahn--Banach theorem (see e.g.\cite[Thm.~IV.3.9~\&~Cor.~IV.3.10]{conway_functional_analysis}) there exists a continuous linear functional $g \in X^{\ast}$ and an $\alpha \in \R$ such that $\langle g, x \rangle < \alpha$ and $\langle g, y \rangle \ge \alpha$ for all $y \in C$. Since the image of the cone $C$ under the functional $g$ can only be one of $\R_{+}$, $\R_{-}$, $\R$ or $\{0\}$, we must have $\alpha = 0$. Hence $g \in C^{\ast}$ but $\langle g, x \rangle < 0$, which implies that $x \notin \tilde C$. Therefore $\tilde C \subseteq C$ also, and so $C = \tilde C$.
		
		Take $x,y \in C$ and let $r \in \R_{+}$ with $rx - y \notin C$. Since $C = \tilde C$, there exists some $f \in C^{\ast}$ such that $\langle f, rx - y \rangle <0$. Consequently, 
		\begin{align}
			\beta(x,y) &= \sup \{ r > 0 \, : \, r x - y \notin C \}
			= \sup \{ r > 0 \, : \, \langle f, rx - y \rangle <0  \text{ for some } f \in C^{\ast} \} \nonumber \\
			&= \sup \{ r > 0 \, : \, r\langle f, x \rangle < \langle f, y \rangle  \text{ for some } f \in C^{\ast} \} 
			= \sup \bigg\{ \frac{\langle f, y \rangle}{\langle f, x \rangle} \, : \,  f \in C^{\ast}, \, \langle f, x \rangle  \neq 0\bigg\}, \label{eq:dual_beta_equiv}
		\end{align}
		and similarly for $\beta(y,x)$. Then \eqref{eq:hilbert_dual} is indeed equivalent to \eqref{eq:hilbert_beta}.
	\end{proof}
	
	\begin{rmk}\label{rmk:triangle}
		Given $ \log \frac{\langle f,y \rangle \langle g,x \rangle}{\langle f,x \rangle \langle g,y \rangle} =  \log \frac{\langle f,z \rangle \langle h,x \rangle}{\langle f,x \rangle \langle h,z \rangle}  +  \log \frac{\langle h,z \rangle \langle g,x \rangle}{\langle h,x \rangle \langle g,z \rangle} $, for any $f,g,h\in C^*$, it is easy to verify the triangle inequality for $\hilbert$ using the representation \eqref{eq:hilbert_dual}.
	\end{rmk}

	\begin{rmk}
		Once more, since the topology on $X$ does not change the Hilbert metric on $C$, when working with \eqref{eq:hilbert_dual} one can choose the topology, and therefore the dual space, cleverly: a coarser topology, with a smaller corresponding dual space, will almost always be preferable.
	\end{rmk}
	
	\subsection{Contraction properties}\label{sec:contraction}
	
	Positive linear operators on a positive closed convex cone contract in the Hilbert projective distance: this result is again due to Birkhoff \cite{birkhoff57}. More generally, for proper closed convex cones $C \subset X$, Birkhoff's contraction theorem can be stated as follows:
	
	\begin{thm}[Birkhoff]\label{thm:birkhoff}
		Let $X$ be a LCS, take $L \, : \, X \to X$ to be a linear transformation, and suppose that $L(C\setminus\{0\}) \subseteq C\setminus\{0\}$. If the diameter $\Delta(L) = \sup_{x,y \in C\setminus\{0\}} \hilbert (Lx, Ly)$ is finite, then we have
		\begin{equation}
			\hilbert (Lx, Ly) \le \tau(L) \hilbert (x,y), \qquad \forall x,y \in C \setminus \{0\},
		\end{equation}
		and $\tau(L) = \tanh \Big( \frac{\Delta(L)}{4} \Big)$ is called the Birkhoff contraction coefficient.
	\end{thm}
	
	The theorem holds equivalently if one drops finiteness of $\Delta(L)$ as a condition and extends the definition of the contraction coefficient to $\tau(L) = 1$ when $\Delta(L) = \infty$. In other words, any bounded linear operator $L$ is non-expansive in $C$ under the Hilbert distance, but it is \textit{strictly contracting} if and only if the diameter $\Delta(L)$ of $C$ under $L$ in the Hilbert metric is finite, i.e. $\tau(L) < 1$.
	
	We will now show that a stronger result than Theorem~\ref{thm:birkhoff} is possible. 
	\begin{defn}[$\mathcal{T}$-distance]\label{def:t-distance}
		For $x, y \in C \setminus \{0\}$, where $C$ is a proper closed convex cone in a real LCS, we define the \textit{hyperbolic tangent of the Hilbert pseudo-metric} as
		\begin{equation}\label{eq:T_distance}
			\mathcal{T} (x,y) := \tanh \Big(  \frac{\hilbert (x,y)}{4}  \Big),
		\end{equation}
		where $\tanh(\infty) : = 1$. For simplicity we refer to \eqref{eq:T_distance} as the \textit{$\mathcal{T}$-distance}.
	\end{defn}
	Note that the $\mathcal{T}$-distance is a pseudo-metric for $C$: one can easily check that the triangle inequality and symmetry properties are inherited from the Hilbert distance. However, the metric $\mathcal{T}$ makes the cone $C$ into a \emph{bounded} space, while $\hilbert$ gives an infinite distance between any points in $\mathring{C}$ and $\partial C$.  Borrowing ideas from the proof of \cite[Thm.~2.3]{dubois09}, we obtain the following theorem.
	\begin{thm}\label{thm:tanh_contaction}
		Let $X$ be a LCS, and $L \, : \, X \to X$ a linear transformation. Suppose that $L(C\setminus\{0\}) \subseteq C\setminus\{0\}$. Then we have
		\begin{equation}\label{eq:tanh_contraction}
			\mathcal{T} (Lx, Ly) \le \tau(L) \mathcal{T} (x,y), \qquad \forall x,y \in C\setminus \{0\},
		\end{equation}
		where $\tau(L) = \sup_{x,y \in C\setminus \{0\}} \mathcal{T}(Lx, Ly)$ is the diameter of $C$ under $L$ in the $\mathcal{T}$-distance, and is equal to the Birkhoff contraction coefficient.
	\end{thm}
	
	\begin{proof}
		For all $x,y \in C \setminus \{0\}$, define the set
		\begin{equation*}
			E_{C}(x,y) := \{  r>0 \, : \, rx - y \notin C \}.
		\end{equation*}
		Using the same notation as in \eqref{eq:hilbert_beta}, we have $\hilbert(x,y) = \log (\beta(x,y) \beta(y,x))$, where
		\begin{equation*}
			\beta(x,y) := \sup E_{C}(x,y) \in (0, \infty], \qquad
			\beta(y,x) := \sup E_{C}(y,x) \in (0, \infty].
		\end{equation*}
		
		Fix $x,y \in C \setminus \{0\}$. If $Lx$ and $Ly$ are collinear, then $\mathcal{T} (Lx, Ly) = 0$ and the claim holds trivially. Similarly, if $\hilbert (x,y) = \infty$, then $\mathcal{T}(x,y) = 1$; as $\mathcal{T} (Lx, Ly)$ is certainly less than its supremum over $x$ and $y$, the claim holds. It remains to consider the case $\hilbert (x,y) < \infty$ and $Lx$ and $Ly$ not collinear (so $\hilbert (Lx, Ly) \neq 0$).
		
		For $r >0$, consider $rx - y \in C$. By linearity of $L$, we also have $rLx - Ly \in C$, and in particular $E_{C}(Lx,Ly) \subset E_{C}(x,y)$. This implies that
		\begin{equation*}
			\beta(x,y) \ge \beta (Lx, Ly) > \frac{1}{\beta(Ly,Lx)} \ge \frac{1}{\beta(y,x)},
		\end{equation*}
		where the strict inequality in the middle is due to \eqref{eq:ineq_beta_beta_inv} and the assumption that $\hilbert(Lx, Ly) \neq 0$. We now approximate $\beta(x,y)$ and $\beta(y,x)$ from above (since $\hilbert(x,y) < \infty$, also $\beta(x,y), \beta(y,x) < \infty$), and $\beta(Lx,Ly)$ and $\beta(Ly, Lx)$ from below, i.e. take $M, m > 0$ and $M',  m' >0$ such that
		\begin{equation*}
			M > \beta(x,y), \quad m > \beta(y,x), \quad  \frac{1}{\beta(Ly, Lx)} < \Big\{ \frac{1}{m'}\, , M'\Big\} < \beta(Lx, Ly).
		\end{equation*}
		
		By definition of $E_C (Lx,Ly)$ and $E_C (Ly,Lx)$, we have $M'Lx-Ly \notin C$, and $m'Ly-Lx \notin C$. Similarly, $Mx - y \in C$ and $my-x \in C$. For $r > 0$, note that
		\begin{equation*}
			r L (Mx - y) - L(my-x) = (rM+1)Lx - (r+m)Ly \in C \Longleftrightarrow \frac{rM+1}{r+m}Lx - Ly \in C.
		\end{equation*}
		Letting $h_1 (r) = (rM+1)/(r+m)$, this implies in particular that
		\begin{align*}
			E_{C}(L (Mx - y), L(my-x))
			&= \Big\{ r > 0 \, : \,  h_1(r)Lx - Ly \notin C  \Big\} \\
			&= \Big\{ h_1^{-1}(w) > 0 \, : \,  w Lx - Ly \notin C  \Big\} \\
			&= h_1^{-1} \bigg( \Big\{ \frac{1}{m}< w<M \, : \,  w Lx - Ly \notin C  \Big\} \bigg) \subset h_1^{-1} \Big( E_{C}(Lx, Ly)    \Big),
		\end{align*}
		where $h_1^{-1}(r) = (rm - 1)/(M-r)$. Since $ M' \in \big( \frac{1}{\beta(Ly,Lx)} \, , \beta(Lx,Ly) \big) \subset (\frac{1}{m}, M)$ by assumption, and $M'Lx - Ly \notin C$, we have that $h_1^{-1}(M') \in E_{C}(L (Mx - y), L(my-x))$.
		
		Analogously, for $r > 0$,
		\begin{equation*}
			r L (my-x) - L(Mx-y)  \in C \Longleftrightarrow \frac{rm+1}{r+M}Ly - Lx \in C.
		\end{equation*}
		Letting $h_2 (r) = (rm+1)/(r+M)$ and $h_2^{-1}(r) = (rM - 1)/(m-r)$, this implies that
		\begin{align*}
			E_{C}(L (my - x), L(Mx-y))
			&= h_2^{-1} \bigg( \Big\{ \frac{1}{M}< w<m \, : \,  w Ly - Lx \notin C  \Big\} \bigg) \subset h_2^{-1} \Big( E_{C}(Ly, Lx)    \Big).
		\end{align*}
		Since $ \frac{1}{m'} \in (\frac{1}{m}, M)$ by assumption, which implies that $m' \in (\frac{1}{M} \, , m)$, and $m'Ly - Lx \notin C$, we have that $h_2^{-1}(m') \in E_{C}(L (my-x), L(Mx-y))$.
		
		We know that $\Delta(L) = \sup_{x,y \in C \setminus \{0\}} \hilbert (Lx, Ly)$ and $Mx-y, my-x \in C$ by definition of $\beta(x,y)$ and $\beta(y,x)$. Therefore,
		\begin{align*}
			h_1^{-1}(M') h_2^{-1}(m') 
			&\le \beta\big( L(Mx-y), L (my - x)\big) \beta\big(L (my - x), L(Mx-y)\big) \\
			&\le \sup_{\tilde x,\tilde y \in C \setminus \{0\}} \beta( L \tilde x, L \tilde y) \beta(L \tilde y, L \tilde x)
			= e^{\Delta(L)},
		\end{align*}
		which yields the inequality
		\begin{equation*}
			\frac{(M'm - 1)}{(M - M')} \frac{(m'M-1)}{(m - m')} \le e^{\Delta(L)}.
		\end{equation*}
		
		Now let $D =\log (M m)$ and $d =\log (M' m')$. Note that since $M' \in (\frac{1}{m}, M)$ and $m' \in (\frac{1}{M}, m)$, $d \le D$. Substituting $m' = \frac{e^d}{M'}$ in the above yields
		\begin{equation}\label{eq:f_M_prime_ineq}
			f(M') := \frac{(M'm - 1)}{(M - M')} \frac{(e^d M-M')}{(m M'- e^d)} \le e^{\Delta(L)}.
		\end{equation}
		Noting that $M' = \frac{e^d}{m'} \in (\frac{e^d}{m}, e^d M )$, intersecting this set with $(\frac{1}{m}, M)$ yields $M' \in (\frac{e^d}{m}, M)$.
		Differentiating the left-hand side of \eqref{eq:f_M_prime_ineq}, we find that the minimum of $f(M')$ within these constraints for $M'$ is attained at $M^{\prime \ast} = e^{d/2} \sqrt{\frac{M}{m}}$. Substituting into the expression above, we get
		\begin{align*}
			f(M^{\prime \ast}) 
			&= \frac{\Big( e^{d/2} \sqrt{Mm} - 1 \Big)}{\Big( 1 - e^{d/2}\frac{1}{\sqrt{Mm}} \Big)} \frac{\Big( e^d -e^{d/2} \frac{1}{\sqrt{Mm}} \Big)}{\Big( e^{d/2} \sqrt{Mm}- e^d \Big)} 
			= \frac{\sinh^2 \Big( \frac{D+d}{4} \Big)}{\sinh^2 \Big( \frac{D-d}{4} \Big)}.
		\end{align*}
		Taking square-roots yields
		\begin{equation*}
			\frac{\sinh\Big( \frac{D+d}{4} \Big)}{\sinh \Big( \frac{D-d}{4} \Big)} \le \sqrt{f(M')} \le e^{\Delta(L)/2}.
		\end{equation*}
		Using the identity $\sinh(a\pm b) = \sinh(a) \cosh(b) \pm \sinh(b) \cosh(a)$ and the fact that $\frac{x-1}{x+1}$ is increasing for $x>0$, we obtain the final expression
		\begin{equation*}
			\tanh \Big( \frac{d}{4}  \Big) \le \tanh \Big( \frac{\Delta(L)}{4}  \Big) \tanh \Big( \frac{D}{4}  \Big).
		\end{equation*}
		Taking limits as $M, m \rightarrow \beta(x,y), \beta(y,x)$ and $M', m'\rightarrow \beta(Lx,Ly), \beta(Ly,Lx)$, we are done.
	\end{proof}
	
	\begin{rmk}
		Birkhoff's Theorem~\ref{thm:birkhoff} is immediate from concavity and monotonicity of $\tanh(x)$ for $x \ge 0$. The advantage of using $\mathcal{T}$ instead of $\hilbert$ is negligible when the distances are small, since $\mathcal{T}(x,y)$ is equivalent to $\hilbert(x,y)$ asymptotically as $\hilbert(x,y)$ approaches 0. However, we can find points $x,y \in C$ such that $\hilbert(x,y) = \infty$, such as when comparing an element $x \in \mathring{C}$ with an element $y \in \partial C$. In these cases, the $\mathcal T$-distance is preferable, since $\mathcal{T}(x,y)$ stays finite and the inequality \eqref{eq:tanh_contraction} remains meaningful.
	\end{rmk}
	
	\begin{exmp}
		As we mentioned in the introduction, an immediate application of Birkhoff's theorem is in the ergodic theory of Markov processes, since transition operators are positive linear operators that map probability distributions to probability distributions, and so the assumptions of Theorem~\ref{thm:birkhoff} are satisfied. We discuss explicit forms of Birkhoff's contraction coefficient for stochastic matrices and a class of transition kernels in Section~\ref{sec:geometry}, and explore the ergodicity of Markov chains specifically in Appendix~\ref{app:markov}.
	\end{exmp}
	
	\section{$\hilbert$-metric on the space of probability measures}\label{sec:hilbert_metric_measures}
	
	From general LCS we now move to the space of probability measures, and consider the Hilbert projective distance in this context specifically. In \cite[Chapter~XVI]{birkhoff_lattice} Birkhoff works with positive cones in a Banach lattice. Since the probability measures are a subset of the positive measures, which form the positive cone in the space of signed measures, the definition of the Hilbert distance on probability measures can be easily deduced from Birkhoff's work (see e.g.~\cite[Eq.~9]{atar-zeitouni97} and \cite[Def.~3.3]{legland04}). In this section we choose to derive the Hilbert distance in the framework of duality instead, drawing a parallel with the works on convex cones \cite{dubois09,rugh10}. The purpose of this exercise is to gain an understanding of the Hilbert metric in terms of functions acting on probability measures, and to investigate how a change in the test-functions affects the distance itself. We note in passing that Eckstein \cite{eckstein24} gives a slightly different application of Hilbert metrics, to cones related to probability measures, in order to encode various integrability conditions.
	
	\begin{notation}
		For any $\sigma$-algebra $\mathcal{F}$ and space $F$, let $L^0(\mathcal{F}, F)$ denote the space of $\mathcal{F}$-measurable functions, valued in $F$, and let $B(\mathcal{F}, F)$ be the subspace of bounded $\mathcal{F}$-measurable functions. For any two spaces $E,F$, let $C_{b}(E,F)$ denote the space of bounded continuous functions $E \to F$. If $E$ and $F$ are metric spaces, let $C_{bL}(E,F)$ be the space of bounded $F$-valued Lipschitz functions. By $\| f \|_{\infty}$ we denote the $L^{\infty}$-norm of $f$ and by $\| f \|_{Lip}$ its Lipschitz coefficient.
	\end{notation}
	
	Let $(E, \mathcal{F})$ be a measurable space, and consider the space $\mathcal{M}(E)$ of finite signed measures on $(E, \mathcal{F} )$. A natural approach is to make $\mathcal{M}(E)$ into a Banach space by equipping it with the total variation norm $\| \cdot \|_{\tvnorm}$. The total variation norm is defined, as usual, by
	\begin{equation}\label{eq:total_var_hahn_jor}
		\| \mu\|_{\tvnorm} := |\mu|(E) = \mu^{+}(E) + \mu^{-}(E), \quad \text{for } \mu\in \mathcal{M}(E),
	\end{equation}
	where $\mu = \mu^{+} - \mu^{-}$ is the Hahn--Jordan decomposition of $\mu$. It can be expressed equivalently in terms of $\mu$ acting on elements of $B(\mathcal{F}, \R)$ as
	\begin{equation}\label{eq:total_var_norm_bm}
		\| \mu\|_{\tvnorm} := \sup \Big\{ \int_{E} f \di \mu \, : \, f \in B(\mathcal{F}, \R), \, \| f \|_{\infty} \le 1 \Big\}.
	\end{equation}
	Now, the subset of positive measures $\mathcal{M}_{+}(E)$ is a proper closed convex cone in $\mathcal{M}(E)$. The probability measures on $( E, \mathcal{F} )$, denoted by $\mathcal{P}(E)$, are a subset of $\mathcal{M}_{+}(E)$. In \eqref{eq:hilbert_dual}, following ideas from \cite{dubois09,rugh10}, we provided an equivalent definition of the Hilbert metric using duality. This is not a convenient approach when viewing $\mathcal{M}(E)$ as the Banach space $\big(\mathcal{M}(E), \| \cdot \|_{\tvnorm}\big)$: for one thing, when dealing with signed measures, one usually prefers to work with a \textit{predual} space instead of the dual.
	
	Taking the predual point of view, we could consider $\mathcal{M}(E)$ as a subset of $X=C_b(E, \R)^*$, which is a real Banach space under the operator norm. When $E$ is a Polish space with the Borel $\sigma$-algebra $\mathcal{B}(E)$, this amounts to a linear isometric embedding that is weak$^{\ast}$-dense. Equipping $X$ with the weak$^{\ast}$-topology, rather than the operator norm, we get that $X$ is a LCS and $X^{\ast}=C_b(E, \R)$, so one expects a predual characterisation of the cone of positive measures in terms of $C_b(E, \R)$. This, however, does not immediately follow from the procedure that led to \eqref{eq:hilbert_dual}. Instead, we give here a direct argument for the desired characterization \eqref{eq:pos_measure_cb_func}, where one can think of $C_b(E,\mathbb{R}_+)$ as the `predual cone', in analogy with \eqref{eq:dual_cone}. In fact, we can further restrict the space to bounded Lipschitz functions.
	\begin{prop}\label{prop:positive_measures_and_cbf}
		Let $E$ be a Polish space with Borel $\sigma$-algebra $\mathcal{B}(E)$ and let $C_{bL}(E,\mathbb{R}_+)$ be the space of bounded Lipschitz functions $E \to \R_+$. Then we can characterize the space of positive measures $\mathcal{M}_+(E)$ as
		\begin{equation}\label{eq:pos_measure_cb_func}
			\mathcal{M}_+(E) = \bigl\{ \mu \in \mathcal{M}(E): \langle \mu , f \rangle \geq 0, \; \: \forall f \in C_{bL}(E,\mathbb{R}_+) \bigr\}.	
		\end{equation}
		If $E$ is only metrizable, then replace $C_{bL}(E,\R_+)$ with $C_b(E,\R_+)$.
	\end{prop}
	\begin{proof}
		Let $M$ denote the right-hand side of \eqref{eq:pos_measure_cb_func}. We want to show that $\mathcal{M}_+(E) = M$. By non-negativity of the elements of $\mathcal{M}_+(E)$ and $C_{bL}(E,\mathbb{R}_+)$ (resp.~$C_{b}(E,\mathbb{R}_+)$), it is obvious that $\mathcal{M}_+(E) \subseteq M$. For the opposite inclusion, suppose $\mu \notin \mathcal{M}_+(E)$. By the Hahn--Jordan decomposition, there exist disjoint sets $P, N \subset E$ such that $P \cup N = E$, and (Borel) measures $\mu^+$ and $\mu^-$ such that $\mu^+$ is supported on $P$ and $\mu^-$ is supported on $N$. Since $\mu \notin \mathcal{M}_+$, we have $\mu^-(N)>0$. Recall that $\mu, \mu^+$ and $\mu^-$ are regular, as they are Borel measures on a metric space $E$ (see e.g~\cite[Thm.~7.1.7]{bogachev}). Take $0 < \varepsilon < \mu^-(N)/4$. By regularity of $\mu^-$, we can find a closed set $A^-_{\varepsilon} \subset N$ such that $\mu^-(N \setminus A^-_{\varepsilon}) < \varepsilon$. Likewise, there exists a closed set $A^+_{\varepsilon} \subset P$ such that $\mu^+(P \setminus A^+_{\varepsilon}) < \varepsilon$. Note that $A^+_{\varepsilon} \cap A^-_{\varepsilon} = \emptyset$, since they are respectively the subsets of disjoint sets $P$ and $N$. For $E$ Polish, we can take the sets $A^-_{\varepsilon},A^+_{\varepsilon}$ to be compact (again, \cite[Thm.~7.1.7]{bogachev}), so the Lipschitz version of Urysohn's Lemma \cite[Prop.~2.1.1]{cobzas_lipschitz} (resp.~Urysohn's Lemma \cite[Thm.~1.2.10]{cobzas_lipschitz}) guarantees that there exists $f \in C_{bL}(E, \R_+)$ (resp.~$C_{b}(E, \R_+)$) taking values in $[0,1]$ with
		\begin{equation*}
			f(x) = \left\{\begin{array}{lr}
				0 & \text{for } x  \in A^+_{\varepsilon},\\
				1 & \text{for } x  \in A^-_{\varepsilon}.
			\end{array} \right.
		\end{equation*}
		Integrating $f$ against $\mu$ we have
		\begin{align*}
			\int_E f \di \mu 
			= \int_E f \di \mu^+ - \int_E f \di \mu^- 
			\le \mu^+(P\setminus A^+_{\varepsilon}) - \mu^-(A^-_{\varepsilon})
			\le 2 \varepsilon - \mu^-(N)
			\le - \frac{\mu^-(N)}{2}
			< 0.
		\end{align*}
		Consequently, we have $f \in C_{bL}(E, \R_+)$ (resp.~$C_{b}(E, \R_+)$), but $\langle \mu, f \rangle <0$, so $\mu \notin M$. Therefore $M \subseteq \mathcal{M}_+(E)$, and the two sets are equal.
	\end{proof}
	
	When $E$ is a Polish space, the above proposition is all we need to define the Hilbert projective (pseudo-)metric on $\mathcal{M}_+(E)$ in terms of bounded positive Lipschitz functions in $C_{bL}(E,\mathbb{R}_+)$, analogously to \eqref{eq:hilbert_dual}. On the other hand, if we do not want to assume $E$ to be Polish, or even metric, we need to enlarge the set of test-functions for the construction of the Hilbert metric to still make sense. Similar to Proposition~\ref{prop:positive_measures_and_cbf}, we find the following (trivial) characterization of $\mathcal{M}_+(E)$ in terms of bounded positive measurable functions.
	
	\begin{prop}\label{prop:positive_measures_and_bmf}
		Let $(E, \mathcal{F})$ be a measurable space and let $B(\mathcal{F},\mathbb{R}_+)$ be the space of $\mathcal{F}$-measurable bounded functions taking values in $\R_+$. Then we have
		\begin{equation}\label{eq:pos_measure_bm_func}
			\mathcal{M}_+(E) = \bigl\{ \mu \in \mathcal{M}(E): \langle f, \mu \rangle \ge 0, \; \: \forall f \in B(\mathcal{F}, \R_{+}) \bigr\}.	
		\end{equation}
	\end{prop}
	\begin{proof}
		Let $M'$ be the right-hand side of \eqref{eq:pos_measure_bm_func}. By non-negativity of the functions in $B(\mathcal{F}, \R_{+})$, $\mu \in \mathcal{M}_{+}(E)$ implies $\mu \in M'$. Conversely, assume $\mu \notin \mathcal{M}_{+}(E)$. Using the Hahn--Jordan decomposition, take $N \in \mathcal{F}$ such that $\mu(N) = - \mu^-(N) <0$. Let $ f : = \mathds{1}_{N} \in B(\mathcal{F}, \R_{+})$. Then $\langle f, \mu \rangle < 0$, but $\langle f, \nu \rangle \ge 0$ for all $\nu \in \mathcal{M}_{+}(E)$. Hence $\mu \notin M'$, and we are done.
	\end{proof}
	
	\begin{prop}
		Let $(E, \mathcal{F})$ be a measurable space. Write $\mathcal{S}=C_{bL}(E,\R_+)$ if $E$ is Polish (with $\mathcal{F}$ the corresponding Borel $\sigma$-algebra), $\mathcal{S}=C_{b}(E,\R_+)$ if $E$ is metrizable, or $\mathcal{S}=B(\mathcal{F}, \R_{+})$ otherwise. Then the Hilbert projective pseudo-metric can be written as follows: for $\mu, \nu \in \mathcal{M}_+(E)\setminus \{0\}$,
		\begin{equation}\label{eq:hilbert_bm}
			\hilbert (\mu,\nu) = \sup_{\substack{f,g \in \mathcal{S} \\\langle f,\mu \rangle, \langle g,\nu \rangle \neq 0} } \bigg\{ \log \frac{\langle f,\nu \rangle \langle g,\mu \rangle}{\langle f,\mu \rangle \langle g,\nu \rangle} \bigg\}.
		\end{equation}
	\end{prop}
	\begin{proof}
		Consider any $\mu, \nu \in \mathcal{M}_+(E)$, and take $r \in \R_{+}$ such that $r\mu - \nu \notin \mathcal{M}_+(E)$. Using Proposition~\ref{prop:positive_measures_and_cbf} and Proposition~\ref{prop:positive_measures_and_bmf}, there is an $f \in \mathcal{S}$ such that $\langle f, r\mu - \nu \rangle <0$. Then a calculation analogous to \eqref{eq:dual_beta_equiv} gives equivalence between \eqref{eq:hilbert_bm} and the original definition of the Hilbert metric \eqref{eq:hilbert_beta}.
	\end{proof}
	
	Now, a natural question to ask is under which conditions $\hilbert(\mu,\nu)$ is finite. For $\mu,\nu \in \mathcal{M}_+(E)$, let
	\begin{equation}\label{eq:beta_predual}
		\beta(\mu,\nu) = \sup_{\substack{f \in B(\mathcal{F}, \R_+)\\ \langle f, \mu \rangle \neq 0}}
		\bigg\{ \frac{\langle f, \nu \rangle}{\langle f, \mu \rangle} \bigg\}.
	\end{equation}
	Clearly, $H(\mu, \nu) < \infty$ if and only if $\beta(\mu,\nu), \beta(\nu,\mu) < \infty$. We see immediately that if there exists an unbounded measurable function $h \in L^0(\mathcal{F},\R_+)$ such that $\langle h, \mu \rangle < \infty$ but $\langle h, \nu \rangle = \infty$, then we can take a sequence of bounded functions $h_n \in B(\mathcal{F}, \R_+)$ such that $h_n \to h$, and the right-hand side of \eqref{eq:beta_predual} is infinite. Consequently, if, for example, $\nu$ has a strictly smaller number of finite moments than $\mu$, then $\beta (\mu,\nu) = \infty$, and conversely if $\nu$ has (strictly) more finite moments, then $\beta(\nu,\mu) = \infty$. Thus, to have $\hilbert (\mu,\nu) < \infty$, we need a condition on $\mu, \nu$ that is quite a lot stronger than simple equivalence of measures (which we denote as usual by $\sim$). This condition, which we call \textit{comparability} again, in accordance with Birkhoff, and denote by $\comp$, has already been stated in \cite[Def.~3.1]{legland04} and \cite{atar-zeitouni97}. Here we derive it directly from the `predual' formulation \eqref{eq:hilbert_bm}.
	
	Let $\mu \sim \nu$, with Radon-Nikodym derivatives $\tfrac{\di \mu}{\di \nu}$ and $\tfrac{\di \nu}{\di \mu}$. For all $\varphi \in B(\mathcal{F}, \R_+)$ we have $\langle \varphi, \nu \rangle  = \langle \varphi \frac{\di \nu}{\di \mu}, \mu \rangle \le \| \frac{\di \nu}{\di \mu}  \|_{L^{\infty}(\mu)} \langle \varphi, \mu \rangle$, so \begin{equation*}
		M_{\varphi} :=\sup_{\substack{\varphi \in B(\mathcal{F}, \R_+)\\ \langle \varphi, \mu \rangle \neq 0}} \bigg\{ \frac{\langle \varphi, \nu \rangle}{\langle \varphi, \mu \rangle} \bigg\}\le \left\| \frac{\di \nu}{\di \mu}  \right\|_{L^{\infty}(\mu)}.
	\end{equation*}
	On the other hand,
	\begin{equation*}
		\Big\langle \varphi \frac{\di \nu}{\di \mu}, \mu \Big\rangle
		= \langle \varphi, \nu  \rangle
		\le \langle \varphi, \mu \rangle M_{\varphi},
	\end{equation*}
	so $\frac{\di \nu}{\di \mu} \le M_{\varphi}$ $\mu$-a.e., and in particular $\big\| \frac{\di \nu}{\di \mu}  \big\|_{L^{\infty}(\mu)} \le M_{\varphi}$.
	Therefore, $\beta(\mu,\nu) =  \big\| \tfrac{\di \nu}{\di \mu}  \big\|_{L^{\infty}(\mu)}$, and hence we can state our comparability condition as follows.
	\begin{defn}
		Let $(E, \mathcal{F})$ be a measurable space. Two positive measures $\mu, \nu \in \mathcal{M}_+(E)$ are \textit{comparable} if $\mu \sim \nu$ and their Radon-Nikodym derivatives $\frac{\di \mu}{\di \nu}$ and $\frac{\di \nu}{\di \mu}$ are essentially bounded, i.e.~ $\frac{\di \mu}{\di \nu} \in L^\infty(\nu)$ and $\frac{\di \nu}{\di \mu} \in L^\infty(\mu)$. Equivalently, $\mu$ and $\nu$ are comparable if there exists scalars $q,r >0 $ such that
		\begin{equation}\label{eq:comparability}
			q\mu(A) \le \nu(A) \le r\mu(A), \quad \forall A \in \mathcal{F}.
		\end{equation}
	\end{defn}
	
	Then the Hilbert projective pseudo-metric for $\mu,\nu \in \mathcal{M}_+(E)$ can be defined as
	\begin{equation}\label{eq:hilbert_indicator_derivatives}
		\hilbert(\mu,\nu) = \log \left( \left\| \frac{\di \mu}{\di \nu}   \right\|_{\infty} \left\| \frac{\di \nu}{\di \mu}   \right\|_{\infty}     \right)
		= \hspace{-5pt} \sup_{\substack{A, B \in \mathcal{F} \\ \nu(A) > 0, \, \mu(B) >0 }} \hspace{-3pt} \bigg\{ \log \frac{\nu(B) \mu(A)}{\mu(B) \nu(A)} \bigg\}, 
		\quad \text{if } \mu \comp \nu,
	\end{equation}
	and $\hilbert(\mu,\nu) = \infty$ otherwise, and these definitions are equivalent to \eqref{eq:hilbert_bm}. The right-most formulation of \eqref{eq:hilbert_indicator_derivatives} is the definition chosen, for example, by Le Gland and Oudjane in \cite[Def.~3.3]{legland04} and Atar and Zeitouni in \cite{atar-zeitouni97}. Note that if $\mu, \nu \in \mathcal{P}(E)$, then $ \big\| \frac{\di \mu}{\di \nu}   \big\|_{\infty}, \big\| \frac{\di \nu}{\di \mu}   \big\|_{\infty} \ge 1$ since $\mu$ and $\nu$ must integrate to 1, and hence also $ \big\| \frac{\di \mu}{\di \nu}   \big\|_{\infty}, \big\| \frac{\di \nu}{\di \mu}   \big\|_{\infty} \le e^{\hilbert(\mu,\nu)}$ by \eqref{eq:hilbert_indicator_derivatives}. Then, for $\mu,\nu \in \mathcal{P}(E)$, $\mu \sim \nu$, and an arbitrary $f \in B(\mathcal{F}, \R)$ such that $\| f \|_{\infty} \le 1$, we have
	\begin{align*}
		| \langle f, \mu - \nu \rangle | 
		&\le  \int_{\big\{ \frac{\di \mu}{\di \nu} \ge 1 \big\}} |f| \Big(  \frac{\di \mu}{\di \nu} - 1  \Big) \di \nu + \int_{\big\{ \frac{\di \mu}{\di \nu} < 1 \big\}} |f| \Big(  1 -\frac{\di \mu}{\di \nu}  \Big) \di \nu \\
		&\le \Big( \Big\| \frac{\di \mu}{\di \nu} \Big\|_{\infty} - 1 \Big) \nu \big(\big\{ \tfrac{\di \mu}{\di \nu} \ge 1 \big\} \big)
		+\Big( 1 - \Big\| \dfrac{\di \nu}{\di \mu} \Big\|_{\infty}^{-1} \Big) \nu \big(\big\{ \tfrac{\di \mu}{\di \nu} < 1 \big\} \big) 
		\le e^{\hilbert (\mu,\nu)} - 1.
	\end{align*}
	Together with the fact that the total variation distance between two probability measures is at most 2, this yields the following bound, first shown by Atar and Zeitouni \cite[Lemma~1]{atar-zeitouni97}:
	\begin{equation}\label{eq:atar_zeitouni_bound}
		\|\mu - \nu \|_{\tvnorm} \le \frac{2}{\log 3} \hilbert (\mu, \nu).
	\end{equation}
	This bound is clearly not sharp, since the right-hand side can easily be much larger than 2. We will improve it in Corollary~\ref{cor:tv_H_bound_prob_measures}. For now, we use Atar and Zeitouni's result to prove the following Lemma.
	
	\begin{lemma}\label{lemma:P_H_complete}
		Let $(E, \mathcal{F})$ be a measurable space. Then $(\mathcal{P}(E), \hilbert)$ is a complete metric space.
	\end{lemma}
	\begin{proof}
		Note that two probability measures $\mu,\nu \in \mathcal{P}(E)$ which are collinear must be necessarily equal, so $\hilbert$ is a metric on $\mathcal{P}(E)$. Let $(\mu_n) \in \mathcal{P}(E)$ be a Cauchy sequence for the Hilbert metric $\hilbert$. Then $(\mu_n) $ is also Cauchy in total variation norm by \eqref{eq:atar_zeitouni_bound}, so $\mu_n \rightarrow \mu \in \mathcal{P}(E)$ in $\| \cdot \|_{\tvnorm}$ since $\mathcal{P}(E)$ is complete as it is a closed subset of $\mathcal{M}(E)$. Since $\langle f, \mu_n \rangle \to \langle f, \mu \rangle$ for $f \in B(E, \R)$ if $\mu_n \to \mu$ in total variation norm, \eqref{eq:hilbert_bm} gives that $\hilbert$ is lower semi-continuous with respect to $\| \cdot \|_{\tvnorm}$, as a supremum over continuous functions. Hence $\hilbert (\mu_n, \mu) \le \liminf_{k \to \infty} \hilbert (\mu_n, \mu_k)$, where the right-hand side goes to $0$ as $n \to \infty$ by the Cauchy assumption. 
	\end{proof}
	
	\begin{corollary}
		Let $(E, \mathcal{F})$ be a measurable space. Then $(\mathcal{P}(E), \mathcal{T})$ is a complete metric space.
	\end{corollary}
	
	We have seen in Section~\ref{sec:defn_and_contraction} that many of the interesting properties of the Hilbert metric also hold for its transformation $\mathcal{T}$. The following gives a key reason why the classic Hilbert pseudo-metric is also of interest:  $\mathcal{H}$ turns the space of probability measures comparable to a reference measure $\rho$ into a normed vector space (with a modified algebra).
	
	\begin{prop}\label{prop:hilbert_is_norm_general}
		Let $(E,\mathcal{F})$ be a measurable space and fix a reference finite measure $\rho\in \mathcal{M}_+(E)$. We let essential infima and suprema be defined with respect to the nullsets of $\rho$, and consider
		\begin{enumerate}[label=(\roman*)]
			\item the space of measures comparable to $\rho$, namely $\mathcal{M}_{\rho} := \{\mu\in \mathcal{M}_+(E): \frac{\di \mu}{\di\rho}, \frac{\di \rho}{\di\mu}\in L^\infty(\rho)\}$, and $\mathcal{P}_\rho := \mathcal{M}_\rho\cap \mathcal{P}(E)$;
			\item  the equivalence relation $\sim_{\mathrm{coll}}$ on $\mathcal{M}_\rho$ given by collinearity, that is $\mu \sim_{\mathrm{coll}} \nu \Leftrightarrow \mu = c\nu$ for some $c>0$; note that $\mathcal{P}_\rho$ is isomorphic to $\mathcal{M}_\rho \big/ \sim_{\mathrm{coll}}$ (as it is a selection of a unique element from each equivalence class);
			\item  the equivalence relation $\sim_{\mathrm{const}}$  on $L^\infty(\rho)$ given by $f\sim_{\mathrm{const}} g\Leftrightarrow f=g+c$  $\rho$-a.e. for some $c\in \mathbb{R}$; and the associated quotient space $\Theta_{\rho}:=L^\infty(\rho)/\sim_{\mathrm{const}}$.
		\end{enumerate}
		Then the map
		\[\|\cdot\|_\Theta: L^\infty(\rho) \to \mathbb{R}; \quad f \mapsto \esssup_{x\in E}f(x) - \essinf_{x\in E}f(x),\]
		defines a seminorm on $L^\infty(\rho)$, and a norm on $\Theta_{\rho}$. Moreover, the map
		\[\theta:\mathcal{M}_\rho\to L^\infty(\rho); \quad \mu \mapsto \log(\di\mu/\di\rho)\] is an isomorphism of the pseudo-metric spaces $(\mathcal{M}_\rho, \hilbert)$ and $(L^\infty(\rho),\|\cdot\|_\Theta)$, satisfying
		\[\mathcal{H}(\mu,\nu) = \|\theta(\mu)-\theta(\nu)\|_\Theta, \qquad \text{ for all }\mu,\nu \in \mathcal{M}_\rho,\]
		and it is an isomorphism of the metric spaces $(\mathcal{P}_\rho, \hilbert)$ and $(\Theta_{\rho} , \|\cdot\|_\Theta)$. In particular, $(\mathcal{P}_\rho, \hilbert)$ is a normed vector space, when endowed with the algebra of (renormalized) addition and scalar multiplication of log-densities.
	\end{prop}
	
	\begin{proof}
		It is easy to see that $\|\cdot\|_\Theta$ is absolutely homogeneous.  From \eqref{eq:hilbert_indicator_derivatives}, we know that
		\begin{align*}
			\hilbert(\mu,\nu) &= \log \left( \left\| \frac{\di \mu}{\di \nu}   \right\|_{\infty}\right) + \log \left( \left\| \frac{\di \nu}{\di \mu}   \right\|_{\infty}     \right)\\
			&= \esssup \bigg\{\log\Big(\frac{\di \mu}{\di \nu}\Big) \bigg\} - \essinf \bigg\{\log\Big(\frac{\di \mu}{\di \nu}\Big) \bigg\}\\
			& = \esssup \bigg\{\theta(\mu)-\theta(\nu) \bigg\} - \essinf \bigg\{\theta(\mu)-\theta(\nu) \bigg\}\\
			&=\|\theta(\mu)-\theta(\nu)\|_\Theta.
		\end{align*}
		From this it follows that $\|\cdot\|_{\Theta}$ is sublinear (as $\hilbert$ satisfies the triangle inequality), and is therefore a seminorm. It is easy to check that $\|\theta\|_{\Theta}=0$ iff $\theta\sim_{\mathrm{const}} 0$, so $\|\cdot\|_{\Theta}$ is a norm on the vector space $\Theta_{\rho}=L^\infty(\rho)/\sim_{\mathrm{const}}$.
		
		For $f \in \Theta_{\rho}$, the inverse of $\theta:\mathcal{M}_\rho\to L^\infty(\rho)$ is given by
		\begin{equation*}
			\theta^{-1}(f)(A) = \int_A \exp(f(x))\di\rho, \qquad \forall A \in \mathcal{F},
		\end{equation*}
		so $\theta$ is clearly a bijection, and hence an isomorphism of $(\mathcal{M}_\rho, \hilbert)$ and $(L^\infty(\rho),\|\cdot\|_\Theta)$.  Similarly, taking account of the equivalence relation, the inverse of $\theta:\mathcal{P}_\rho \to \Theta_{\rho}$ is given, for $f\in \Theta_\rho$, by
		\begin{equation*}
			\theta^{-1}(f+c)(A) = \frac{\int_A \exp(f(x))\di\rho}{\int_E \exp(f(x))\di\rho}, \qquad \forall A \in \mathcal{F},
		\end{equation*}
		which clearly does not depend on the choice of $c\in \mathbb{R}$ (and hence is well defined on $\Theta_{\rho}=L^\infty(\rho)/\sim_{\mathrm{const}}$). It follows that $\theta$ is an isomorphism of $(\mathcal{P}_\rho, \hilbert)$ and $(\Theta_{\rho}, \|\cdot\|_\Theta)$.
		
		Finally, as $(\Theta_{\rho}, \|\cdot\|_\Theta)$ is a normed vector space, we simply observe that addition and scalar multiplication in $\Theta_{\rho}$ correspond to (renormalized) addition and scalar multiplication of log-densities.
	\end{proof}
	
	\begin{rmk}\label{rmk:quotient_in_V}
		We will see in Section~\ref{sec:geometry} that, when $E$ is finite, we can avoid the equivalence relation above by selecting the unique representative $\theta_0(\mu)$ which satisfies $\theta_0(\mu)(x_0)=0$ for a fixed $x_0\in E$ (see Remark~\ref{rmk:choice_of_theta_zero}). This does not work as cleanly in infinite state spaces, as the value at a single point is typically not well defined when functions are only specified $\rho$-a.e.
	\end{rmk}
	
	\begin{rmk}
		Proposition~\ref{prop:hilbert_is_norm_general} also helps us to understand the topology of $(\mathcal{P}(E), \hilbert)$. For every $\rho\in \mathcal{P}(E)$, we have the corresponding vector space $\mathcal{P}_\rho$ (and any $\rho'\in \mathcal{P}_\rho$ will yield $\mathcal{P}_{\rho'}=\mathcal{P}_\rho)$. As they are normed vector spaces (with an appropriate algebra), these sets are both closed and open, and give a disconnected partition of $\mathcal{P}(E)$. In other words, $(\mathcal{P}(E), \hilbert)$ has the topology of a disjoint union of normed vector spaces (which may have different dimensions). 
	\end{rmk}
	
	We conclude this (rather lengthy) section about the Hilbert metric on probability measures with a few important observations, which motivate why we started looking carefully at the `predual' formulation of the Hilbert metric in the first place.
	\begin{rmk}
		Take $\mu,\nu \in \mathcal{M}_+(E)$, with $E$ Polish. Consider distances of the form
		\begin{equation*}
			D(\mu, \nu) = \sup \Big\{ \Big| \int_E f \di (\mu-\nu) \Big| \, : \, f \in X   \Big\}, \quad X \subseteq C_{bL}(E, \R_+).
		\end{equation*}
		Different conditions on $\| f \|_{\infty}$ and $\| f \|_{Lip}$  yield different metrics: the total variation norm \eqref{eq:total_var_norm_bm} if one imposes $\| f \|_{\infty} \le 1$, the bounded-Lipschitz distance by taking $\| f \|_{\infty} + \| f \|_{Lip} \le 1$, or the 1-Wasserstein distance $\mathcal{W}_1$ (when $\mu$ and $\nu$ are additionally taken to have finite first moment) by imposing $\| f \|_{Lip} \le 1$. This differentiation based on the choice of test-functions is \textit{completely lost} when we work with the Hilbert metric. If we restricted our space to $\mathcal{P}_1(E)$ (the probability measures on $E$ with finite first moment), for example, we could of course characterize our `predual' cone of test-functions using (unbounded) positive Lipschitz functions $Lip(E, \R_+)$, analogously to the Kantorovich--Rubinstein dual formulation of $\mathcal{W}_1$. However, this would not yield a different metric from \eqref{eq:hilbert_bm}. Since any Lipschitz function can be approximated from below by bounded Lipschitz functions, if $\mu \comp \nu$ and $\mu,\nu \in \mathcal{P}_1(E)$, taking the supremum over $Lip(E, \R_+)$ or $C_{bL}(E, \R_+)$ does not change the Hilbert distance.
	\end{rmk}
	
	In the wake of the above remark, we deduce that the Hilbert metric is stronger than the $p$-Wasserstein distance $\mathcal{W}_p$.
    \begin{prop}
		Let $(E, d)$ be a metric space, and let $\{\mu_n\} \in \mathcal{P}_{p}(E)$ be a sequence of probability measures with finite $p^{\text{th}}$-moment converging to $\mu \in \mathcal{P}(E)$ in $\hilbert$. Then $\mathcal{W}_p(\mu_n, \mu) \to 0$ as $n \to \infty$ also.
    \end{prop}
    \begin{proof}
		By \eqref{eq:atar_zeitouni_bound}, convergence in Hilbert metric implies convergence in total variation norm, which in turn implies $\mu_n \to \mu$ weakly. Moreover, $\mu \in \mathcal{P}_p(E)$, by definition of the Hilbert metric and comparability of measures \eqref{eq:comparability}. Fix an arbitrary $x_0 \in E$. Then an argument similar to the one that lead to \eqref{eq:atar_zeitouni_bound} yields that, for all $q \le p$,
		\begin{equation}\label{eq:W_p_convergence}
			\bigg|  \int_{E} d(x_0,x)^q \di (\mu_n - \mu) \bigg|
			\le K_q \big( e^{\hilbert(\mu_n, \mu)} -1 \big),
		\end{equation}
		where $K_q < \infty$ is the $q^{\text th}$-moment of $\mu \in \mathcal{P}_p(E)$. So convergence of moments is preserved under convergence in the Hilbert metric, and thus convergence in the Hilbert metric implies convergence in $\mathcal{W}_p$.
    \end{proof}
	
	The Kantorovich--Rubinstein dual formulation in particular yields the following bound for the $\mathcal{W}_1$ distance with respect to $\hilbert$.
	\begin{corollary}
		For $\mu,\nu \in \mathcal{P}_1(E)$, and an arbitrary $x_0 \in E$, we have
		\begin{equation}\label{eq:W_1_bound_hilbert}
			\mathcal{W}_1(\mu, \nu) \le \big( e^{\hilbert(\mu, \nu)} -1 \big) \, \min \left\{ \int_E d(x_0,x) \di \mu , \int_E d(x_0,x) \di \nu \right\}.
		\end{equation}
	\end{corollary}
	
	Finally, our work so far, the definition of the Hilbert metric and of comparability of measures \eqref{eq:comparability}, all clearly emphasise that convergence in the Hilbert metric is a very strong form of convergence. The Hilbert projective metric not only dominates $\tvnormbig$ and $\mathcal{W}_p$, but also the Kullback--Leibler divergence (or relative entropy):
	\begin{equation}
		D_{KL}(\mu \| \nu ) : = \int_E \log \frac{\di \mu}{\di \nu} \di \mu \le \log \Big\| \frac{\di \mu}{\di \nu} \Big\|_{\infty} \le \hilbert (\mu,\nu).
	\end{equation}
	
	In fact, one can show that the Hilbert metric dominates all $f$-divergences.
	\begin{defn}[$f$-divergence]\label{defn:f_divergence}
		Let $f: \, \R_+ \to (-\infty, \infty]$ be a convex function with $f(1) = 0$, and $f(x)<\infty$ for all $x>0$. Let $\mu,\nu \in \mathcal{P}(E)$, with $\mu \ll \nu$. Then the $f$-divergence of $\mu$ from $\nu$, denoted by $D_f(\mu \| \nu)$, is given by
		\begin{equation}\label{eq:f_divergence_def}
			D_{f}(\mu \| \nu) = \int_{E} f \Big( \frac{\di \mu}{\di \nu}   \Big) \di \nu.
		\end{equation}
	\end{defn}
	\begin{exmp}
		Total variation distance, Kullback--Leibler divergence, Jensen--Shannon divergence, squared Hellinger distance, $\alpha$-divergence and $\chi^2$-divergence are all examples of $f$-divergences.
	\end{exmp}
	\begin{prop}
		Let $(E, \mathcal{F})$ be a measurable space, and let $\{\mu_n\} \in \mathcal{P}(E)$ be a sequence of probability measures converging to $\mu \in \mathcal{P}(E)$ in $\hilbert$. Then, for any $f$-divergence $D_f$, $D_{f}(\mu_n \| \mu) \to 0$ and $D_{f}(\mu \| \mu_n) \to 0$ as $n \to \infty$ also.
	\end{prop}
	\begin{proof}
		Let $f$ be a convex function of the form specified in Definition~\ref{defn:f_divergence}, and let $D_f(\mu \| \nu)$ be the associated $f$-divergence of $\mu$ from $\nu$, where $\mu,\nu \in \mathcal{P}(E)$ and $\mu \ll \nu$. Note that $D_f$ is unchanged if we add a linear term to $f$, i.e. let $\bar f(u) = f(u) + c(u-1)$, then $D_{\bar f}(\mu \| \nu) = D_{f}(\mu \| \nu)$. Moreover, by taking $c \in - \partial f(1)$ (where by $\partial$ we denote the subgradient of $f$), we have $0 \in \partial \bar f(1)$, so without loss of generality we can restrict our attention to convex functions $f$ such that $f(1) = 0 \in \partial f(1)$.
		
		Consider a sequence $\{ \mu_n \} \in \mathcal{P}(E)$ such that $\lim_{n \to \infty} \hilbert(\mu_n, \mu) = 0$. Then there exists $N>0$ such that for all $n\ge N$, $\mu_n \sim \nu$. Since $f$ must be decreasing for $x<1$ and increasing for $x>1$ by virtue of being convex, we have, for all $n \ge N$,
		\begin{align*}
			D_{f}(\mu_n \| \mu)
			&= \int_{\big\{\frac{\di \mu_n}{\di \mu} \le 1\big\}} f \Big( \frac{\di \mu_n}{\di \mu}   \Big) \di \mu + \int_{\big\{\frac{\di \mu_n}{\di \mu} > 1\big\}} f \Big( \frac{\di \mu_n}{\di \mu}   \Big) \di \mu \\
			&\le f \Big(\essinf_{x \in E} \frac{\di \mu_n}{\di \mu}\Big) \mu \big( \big\{ \tfrac{\di \mu_n}{\di \mu} \le 1 \big\} \big) +
			f \Big(\esssup_{x \in E} \frac{\di \mu_n}{\di \mu}\Big) \mu \big( \big\{ \tfrac{\di \mu_n}{\di \mu} > 1 \big\} \big) \\
			&\le \max \bigg\{ f \Big( \Big\| \frac{\di \mu}{\di \mu_n}  \Big\|^{-1}_{\infty} \Big), \, f \Big( \Big\| \frac{\di \mu_n}{\di \mu}  \Big\|_{\infty} \Big) \bigg\}
			\le \max \Big\{ f \big( e^{- \hilbert (\mu_n,\mu)} \big), \, f \big(  e^{\hilbert (\mu_n,\mu)}  \big) \Big\},
		\end{align*}
		where we have used that $\big\| \tfrac{\di \mu}{\di \mu_n}  \big\|_{\infty}, \big\| \tfrac{\di \mu_n}{\di \mu}  \big\|_{\infty} \ge 1$, and $\big\| \tfrac{\di \mu}{\di \mu_n}  \big\|^{-1}_{\infty} \ge e^{-\hilbert(\mu_n,\mu)}$ and $\big\| \tfrac{\di \mu_n}{\di \mu}  \big\|_{\infty} \le e^{\hilbert(\mu_n,\mu)}$. Since $f(1) = 0$ by assumption, the right-hand side of the above goes to 0 as $\hilbert (\mu_n, \mu) \to 0$, so $D_{f}(\mu_n \| \mu)$ converges. The argument for $D_{f}(\mu \| \mu_n)$ is analogous by symmetry, and we are done.
	\end{proof}
	
	\section{Hilbert projective geometry on the probability simplex}\label{sec:geometry}
	
	In this section we consider the Hilbert metric on the probability measures with finite state-space $E \cong \{0,\cdots,n\}$, which form the \textit{probability simplex}. In this case, the form of the Hilbert metric simplifies, and there exists an explicit expression for Birkhoff's contraction coefficient (see Section~4 of \cite[Chapter~3]{seneta_matrices}). We briefly remark on this, and present a short derivation of Birkhoff's coefficient using duality. Then we move on to studying the geometry of the probability simplex under the Hilbert projective metric: using a coordinate transformation inspired by information geometry \cite{amari85,amari16}, we describe the Hilbert balls as convex polytopes in the probability simplex, in an extension of the work in \cite{phadke_hexagons75} to the $n$-dimensional case.
	
	\vspace{5pt}
	
	Let  $E\cong \{0,\cdots, n\}$. The probability measures $\mathcal{P}(E)$ are represented by the set
	\begin{equation*}
		\mathcal{P}(E) \cong \mathcal{S}^n = \bigg\{  x \in \R^{n+1} \, : \, \sum_{i=0}^n x_i = 1 \bigg\} \subset \R^{n+1},
	\end{equation*}
	and $\mathcal{S}^n$ is called the \textit{$n$-dimensional probability simplex}. It is given by the intersection of the convex cone of non-negative vectors $\R_{+}^{n+1}$ with the plane $\sum_{i} x_i=1$. 
	
	Consider the Hilbert distance on the non-negative orthant $C = \R_+^{n}$. Recall the duality expression for $\hilbert$ given in \eqref{eq:hilbert_dual} and the equality \eqref{eq:dual_beta_equiv}. The dual cone $C^{\ast}$ is again $\R^{n}_+$. Take $x,y \in \R_{+}^{n} \setminus \{0\}$ such that $\beta(x,y) <\infty$. Note that $\beta(x,y) <\infty$ if and only if there exists a scalar $ b >0$ such that $y^i \le b x^i$ for all $i = 1, \dots, n$, which in particular implies that $x^i > 0$ whenever $y^i > 0$. Hence we have
	\begin{equation}\label{eq:beta_max_finite_dim}
		\beta(x,y)
		= \sup_{\substack{w \in \R_{+}^{n}\setminus\{0\} \\ w^\intercal x > 0}} \frac{w^\intercal y}{w^\intercal x} 
		= \sup_{r \in [0,1]^{n+1}\setminus\{0\}} \sum_{j:\, x^j > 0} r^j \frac{y^j}{x^j}
		=\max_{j: \, x^j >0} \frac{y^j}{x^j}
		= \max_{j: \, e_j^\intercal x >0} \frac{e_j^\intercal y}{e_j^\intercal x},
	\end{equation}
	where the second equality follows by setting $r^j = w^j x^j / w^\intercal x$, $0 \le r^j \le 1$ for all $j = 1, \dots, n$ and at least one $r^j > 0$, and $\{e_j\}_{j=1}^n$ are the basis vectors of $\R_{+}^{n}$. So $\sup_w w^\intercal y/w^\intercal x$ is attained when $w$ is a basis vector. By symmetry, we have that $\beta(y,x) < \infty$ if there exists $b'>0$ such that $x^i \le b' y^i$ for all $i = 1, \dots n$. Then two elements $x,y \in \R_+^{n}$ are \textit{comparable} (denoted again by $\comp$) if there exist constants $a,b > 0$ such that $ax \le y \le bx$, where the inequalities hold component-wise (this is the definition of comparability given in \cite[Chapter XVI]{birkhoff_lattice}). Then the definition \eqref{eq:hilbert_beta} of the Hilbert projective distance for $x,y \in \R^{n}_+$ simplifies to
	\begin{equation}\label{eq:hilbert_metric_finite}
		\hilbert(x, y) = \left\{ \begin{array}{ll}\vspace{2pt}
			\log \left( \frac{\max_{i : y^i > 0} \frac{x^i}{y^i}}{\min_{j : y^j > 0} \frac{x^j}{y^j}} \right),
			& x \comp y, \\
			\infty, & x \ncomp y.
		\end{array} \right.
	\end{equation}
	
	\begin{rmk}
		In this finite-state context, the comparability condition $\comp$ reduces to \textit{equivalence of measures} $\sim$ on $\mathcal{S}^n$. Recall that $(\mathcal{S}^n, \hilbert)$ is a complete metric space by Lemma~\ref{lemma:P_H_complete}.
	\end{rmk}
	
	\subsection{Explicit forms of the Birkhoff's contraction coefficient}
	
	Using \eqref{eq:beta_max_finite_dim} we can now easily derive an explicit expression for Birkhoff's contraction coefficient of a linear operator $\R^{n}_+ \setminus \{0\} \to \R^{n}_+ \setminus\{0\}$. Define a matrix  $A = (A_{ij})$ to be \textit{strictly positive} (resp. \textit{non-negative}) if $A_{ij} > 0$ (resp. $A_{ij} \ge 0$) for all $i,j$; we will sometime use the notation $A > 0$ (resp. $A \ge 0$).  Let $A$ be \textit{allowable} if $A$ is non-negative and if every row and every column of $A$ has at least one strictly positive element (this definition is given by Seneta in \cite[Def.~3.1]{seneta_matrices}). Clearly any linear operator $\R^{n}_+ \setminus \{0\} \to \R^{n}_+ \setminus\{0\}$ can be represented as an allowable $n \times n$ matrix. We prove the following result, which was already stated by Birkhoff without proof in Corollary~2 of \cite[Chapter~XVI,~Section~3]{birkhoff_lattice} and obtained by Seneta in Section~4 of \cite[Chapter~3]{seneta_matrices}, although the derivation there is significantly more involved.
	\begin{prop}\label{prop:contraction_coeff}
		Let $A = (A_{ij})$ be an allowable $n\times n$ matrix. Birkhoff's contraction coefficient $\tau(A)$ can be written as
		\begin{equation}\label{eq:bir_coeff_matrix}
			\tau(A) =\frac{1 - \sqrt{\phi(A)}}{1 + \sqrt{\phi(A)}}, \quad \text{with } \phi(A) = \min_{i,j,k,l} \, \frac{A_{ik}A_{jl}}{A_{jk}A_{il}},
		\end{equation}
		(with the convention that $0/0 = 1$).
	\end{prop}
	\begin{proof}
		Consider the diameter of $\R_{+}^{n}$ under the matrix $A$, i.e.
		\begin{equation*}
			\Delta(A) = \sup_{x,y \in \R_{+}^{n} \setminus \{0\}} \hilbert (Ax, Ay).
		\end{equation*}
		Assume that $\Delta(A) < \infty$, so $\beta(Ax, Ay), \beta(Ay, Ax) < \infty$ for all $x, y \in \R^{n}_+ \setminus \{0\}$. Note that $\Delta(A) < \infty$ if and only if $A$ is strictly positive, so in particular $Ax$ has strictly positive entries for all $x \in \R^{n}_+$, which implies that $w^\intercal A x >0$ for all $x,w \in \R^n_+$. Using \eqref{eq:beta_max_finite_dim} in the second and fourth equalities below, we get 
		\begin{align*}
			e^{\Delta(A)}
			&= \sup_{x,y\in \R_{+}^{n}\setminus\{0\}} \sup_{w,z\in \R_{+}^{n}\setminus\{0\}} \bigg\{ \frac{w^\intercal Ay }{w^\intercal Ax}\frac{z^\intercal Ax }{z^\intercal Ay} \bigg\}
			= \sup_{x,y\in \R_{+}^{n}\setminus\{0\}} \max_{i,j} \: \bigg\{ \frac{e_i^\intercal Ay }{e_i^\intercal Ax}\frac{e_j^\intercal Ax }{e_j^\intercal Ay} \bigg\}  \\
			&= \max_{i, j} \sup_{x,y\in \R_{+}^{n}\setminus\{0\}} \bigg\{ \frac{y^\intercal A^\intercal e_i }{y^\intercal A^\intercal e_j }\frac{x^\intercal A^\intercal e_j }{x^\intercal A^\intercal e_i} \bigg\}
			= \max_{i, j} \max_{k, l} \bigg\{ \frac{e_k^\intercal A^\intercal e_i }{e_k^\intercal A^\intercal e_j }\frac{e_l^\intercal A^\intercal e_j }{e_l^\intercal A^\intercal e_i} \bigg\}
			= \max_{i,j,k,l} \, \frac{A_{ik} A_{jl}}{A_{jk} A_{il}},
		\end{align*}
		so finally Birkhoff's contraction coefficient is given by
		\begin{equation}\label{eq:birkhoff_rearrangment_calculation}
			\tau(A)
			= \tanh\left(\frac{\Delta(A)}{4}\right) 
			= \frac{e^{\frac{\Delta(A)}{2}} - 1}{e^{\frac{\Delta(A)}{2}} + 1}
			=\frac{1 - \sqrt{\phi(A)}}{1 + \sqrt{\phi(A)}}, \quad \text{with } \phi(A) = \min_{i,j,k,l} \, \frac{A_{ik}A_{jl}}{A_{jk}A_{il}}.
		\end{equation}
		We can check that if $A$ is not strictly positive, so $\Delta(A) = \infty$, the formula above still holds with the convention that $\tanh(\infty):=1$.
	\end{proof}
	
	In fact, we obtain a similar representation of Birkhoff's contraction coefficient for a class of transition kernels on more general spaces. Let $E$ be a Polish space with Borel $\sigma$-algebra $\mathcal{B}(E)$, and consider a positive kernel $K$ on $\mathcal{B}(E) \times E$. Then there exists an associated positive linear operator (again denoted by $K$) such that $K \, : \, \mathcal{M}_+(E) \to \mathcal{M}_+(E)$ and
	\begin{equation}\label{eq:kernel}
		K \mu(\di a)  = \int_{E} K(\di a, x) \di \mu(x).
	\end{equation}
	In the statement below, we restrict our attention to kernels that have a density with respect to a reference measure $\rho \in \mathcal{M}_+(E)$. In other words, let $K(\di a, x) = \kappa(a,x) \di \rho(a)$ for some positive function $\kappa \, : \, \mathrm{Supp}(\rho) \times E \to \R_+$. Then \eqref{eq:kernel} reduces to
	\begin{equation}\label{eq:kernel_abs_cont}
		K \mu(A)  = \int_{a \in A} \int_{x \in E} \kappa(a,x) \di \mu(x)  \di \rho(a), \quad \forall A \in \mathcal{B}(E),
	\end{equation} 
	(where we have swapped the order of integration using Tonelli's theorem). We then have the following infinite dimensional counterpart to Proposition~\ref{prop:contraction_coeff}.
	
	\begin{prop}\label{prop:bikhoff_coeff_kernel}
		Let $E$ be a Polish space with Borel $\sigma$-algebra $\mathcal{B}(E)$ and reference measure $\rho \in \mathcal{M}_+(E)$ with support $\mathrm{Supp}(\rho) \subset E$. Consider a kernel operator $K:\mathcal{M}_+(E) \to \mathcal{M}_+(E)$ of the form \eqref{eq:kernel_abs_cont} defined by a density
		\begin{equation*}
			\frac{\di (K\mu)}{\di \rho}(a) :=  \int_E \kappa(a,x) \di\mu(x),
		\end{equation*}
		for $\kappa: \mathrm{Supp}(\rho) \times E \to (0,\infty)$. Assume $\kappa$ is a bounded continuous function. Then the Birkhoff coefficient of $K$ is given by
		\begin{equation}\label{eq:birkhoff_coeff_transition}
			\tau(K) =\frac{1 - \sqrt{\phi(K)}}{1 + \sqrt{\phi(K)}}, \quad \text{with}\quad \phi(K) = \inf_{ \substack{ a,b\, \in \, \mathrm{Supp}(\rho) \\ x,y \,\in \, E}} \Big\{\frac{\kappa(a,x)\kappa(b,y)}{\kappa(a,y)\kappa(b,x)} \Big\}.
		\end{equation} 
	\end{prop}
	\begin{proof}
		From the structure of the operator, we know that the Radon--Nikodym derivative of $K\mu$ and $K\nu$ (for $\nu \neq 0$) is given by
		\begin{equation*}
			\frac{\di (K\mu)}{\di (K\nu)}(a) = \frac{ \int_E \kappa(a, x) \di\mu(x)}{\int_E \kappa(a, x) \di\nu(x)}, \quad \forall a \in \mathrm{Supp}(\rho).
		\end{equation*}
		As $\kappa$ is continuous and bounded, by dominated convergence we have that $a\mapsto \frac{\di (K\mu)}{\di (K\nu)}(a)$ is continuous, and so the $\rho$-essential supremum of $\frac{\di (K\mu)}{\di (K\nu)}(a)$ is equal to its (pointwise) supremum on $\mathrm{Supp}(\rho)$.
		Using the definition of $\mathcal{T}$ and \eqref{eq:hilbert_indicator_derivatives}, we know that 
		\begin{align*}
			\mathcal{T}(K\mu, K\nu) &= \tanh\bigg(\frac{1}{4} \log \Big(\sup_a\Big\{\frac{ \int_E \kappa(a, x) \di\mu(x)}{\int_E \kappa(a, x) \di\nu(x)}\Big\}\sup_b\Big\{\frac{ \int_E \kappa(b, y) \di\nu(y)}{\int_E \kappa(b, y) \di\mu(y)}\Big\}\Big)\bigg)\\
			&=\sup_{a,b} \bigg\{\tanh\bigg(\frac{1}{4} \log\Big( \frac{\int_E \kappa(a, x) \di\mu(x)}{\int_E \kappa(b, y) \di\mu(y)} \frac{\int_E \kappa(b, y) \di\nu(y)}{\int_E \kappa(a, x) \di\nu(x)}\Big) \bigg)\bigg\}.
		\end{align*}
		From the definition of $\tau$ in Theorem~\ref{thm:tanh_contaction} and monotonicity of $\tanh$, we know that 
		\begin{equation}\label{eq:tau_estimate_operator}
			\begin{split}
				\tau (K) &= \sup_{\mu,\nu\in \mathcal{M}_+(E)} \mathcal{T}(K\mu, K\nu) \\
				&=\sup_{a,b} \bigg\{\tanh\bigg(\frac{1}{4} \log\Big(\sup_{\mu\in \mathcal{M}_+(E)} \Big\{\frac{\int_E \kappa(a, x) \di\mu(x)}{\int_E \kappa(b, y) \di\mu(y)} \Big\}\sup_{\nu\in \mathcal{M}_+(E)}\Big\{\frac{\int_E \kappa(b, y) \di\nu(y)}{\int_E \kappa(a, x) \di\nu(x)}\Big\}\Big) \bigg)\bigg\}.
			\end{split}
		\end{equation}
		
		In order to compute the inner suprema, we observe that for all $\mu \in \mathcal{M}_+(E)$ and $a,b \in E$,
		\[\frac{\int_E \kappa(a, x) \di\mu(x)}{\int_E \kappa(b, y) \di\mu(y)} = \int_E \frac{\kappa(b,x)}{\int_E \kappa(b, y) \di\mu(y)} \frac{\kappa(a,x)}{\kappa(b,x)} \di\mu(x) = \int_E \frac{\kappa(a,x)}{\kappa(b,x)} \di\tilde\mu_b(x)\]
		where $\tilde \mu_b \in \mathcal{P}(E)$ is defined by $\frac{\di \tilde \mu_b}{\di \mu}(x) = \frac{\kappa(b,x)}{\int_E \kappa(b, y) \di\mu(y)}$. Therefore, as $\kappa$ is continuous, 
		\[\sup_{\mu\in \mathcal{M}_+(E)} \frac{\int_E \kappa(a, x) \di\mu(x)}{\int_E \kappa(b, y) \di\mu(y)}= \sup_{\tilde \mu \in \mathcal{P}(E)} \int_E \frac{\kappa(a,x)}{\kappa(b,x)} \di\tilde\mu(x) = \sup_{x\in E}  \frac{\kappa(a,x)}{\kappa(b,x)}. \]
		Substituting in \eqref{eq:tau_estimate_operator} yields
		\begin{equation*}
			\tau(K) =\tanh\Big(\frac{\Delta(K)}{4}\Big), \quad \text{with}\quad \Delta(K) = \sup_{\substack{ a,b\, \in \, \mathrm{Supp}(\rho) \\ x,y \,\in \, E}} \Big\{\log\Big(\frac{\kappa(a,x)\kappa(b,y)}{\kappa(a,y)\kappa(b,x)} \Big)\Big\},
		\end{equation*}
		and essentially the same calculation as \eqref{eq:birkhoff_rearrangment_calculation} gives the form \eqref{eq:birkhoff_coeff_transition}.
	\end{proof}

    In Appendix \ref{app:markov}, we consider in further detail the computation of the Birkhoff contraction coefficient for the generators of finite Markov chains, in discrete and continuous time.
 
	\subsection{Hexagonal polytopes in Hilbert projective geometry}
	
	We now introduce a change of coordinates from the interior of $\mathcal{S}^n$ to $\R^n$ which allows us to build an understanding of the geometry of the probability simplex as a metric space equipped with the Hilbert distance $\hilbert$. We have already employed this coordinate transformation in \cite{cohen_fausti_23} in the setting of nonlinear filtering, and it proved crucial for the analysis of the Hilbert distance between two probability measures $\mu, \nu \in \mathcal{S}^n$.
	
	\begin{notation}
		For simplicity, let $\mathbf{N} := \{ 0, 1, \dots, n \}$ throughout this section.
	\end{notation}
	
	Let $\mathring{\mathcal{S}}^n$ be the interior of the $n$-dimensional probability simplex $\mathcal{S}^n \subset \R^{n+1}$. Following Amari \cite{amari85,amari16}, we map any discrete distribution $\mu \in \mathring{\mathcal{S}}^n$ to its natural parameters $\theta \in \R^n$. In other words, for all $k = 0, \dots, n$, let $\theta_k \, : \, \mathring{\mathcal{S}}^n \to \R^n$ such that
	\begin{equation}\label{eq:theta_transf}
		\theta_k^i(\mu) = \log \frac{\mu^i}{\mu^k}, \quad \forall i \in \mathbf{N}\setminus \{k\}.
	\end{equation}
	Since we exclude the $k^{\mathrm th}$ component, $\theta_k(\mu) = \{\theta^i_k(\mu)\}_{i\neq k}$ is an $n$-dimensional vector. The inverse mapping $\theta_k^{-1} \, : \, \R^n \to \mathring{\mathcal{S}}^n$ is given by
	\begin{equation}\label{eq:theta_transf_inv}
		\mu^i(\theta_k)  = \frac{e^{\theta_k^i}}{\sum_{i=0}^n e^{\theta_k^i}}, \quad \forall i \in \mathbf{N} \quad (\text{where for notational simplicity }\theta_k^k\equiv 0).
	\end{equation}
	Then $\theta_k$ is a diffeomorphism $\mathring{\mathcal{S}}^n \to \R^n$, and a global chart for $\mathring{\mathcal{S}}^n$. Note that we have $n+1$ choices for $k$, so in fact we have a family of $n+1$ coordinate transformations.
	
	As we already noted in \cite[Eq.~19]{cohen_fausti_23}, we have the convenient equivalence
	\begin{equation}\label{eq:theta_inf_and_hilbert}
		\hilbert(\mu,\nu) = \max_{k} \| \theta_k(\mu) - \theta_k(\nu) \|_{\ell^{\infty}} = \max_{i,k} \big( \theta_k^i(\mu) - \theta_k^i(\nu) \big),
	\end{equation}
	where by $\ell^{\infty}$ we denote the standard supremum norm between vectors. 
	
	\begin{rmk}\label{rmk:choice_of_theta_zero}
		It is informative to compare this with Proposition~\ref{prop:hilbert_is_norm_general}. We fix $k=0$ for simplicity, and $\rho$ as the counting measure on $E$. Writing $\mathbf{1} \in \mathbb{R}^n$ for the vector of ones, we then have
		\begin{align*}
			\theta_0(\mu) &= \big( \log(\mu^1/\mu^0),\log(\mu^2/\mu^0),..., \log(\mu^n/\mu^0) \big) \\
			&= \big( \log(\mu^1),\log(\mu^2),..., \log(\mu^n)\big)  - \log(\mu^0)\mathbf{1}.
		\end{align*}
		On the other hand, in the notation of Proposition~\ref{prop:hilbert_is_norm_general}, we have the equivalence class
		\[\theta(\mu) = \Big\{ \big( \log(\mu^0), \log(\mu^1),..., \log(\mu^n) \big) + c\mathbf{1}; \; c\in \mathbb{R}\Big\} \in \Theta_\rho \cong \mathbb{R}^{n+1}\big/\sim_{\mathrm{const}}.\]
		Of course, we can identify $\theta_0(\mu)$ with $\big(0,\theta_0(\mu) \big) \in \theta(\mu)$. Therefore, we see that our $\theta_0$-coordinates \eqref{eq:theta_transf} simply choose the representative element in $\theta(\mu)$ with $c=-\log(\mu^0)$, or equivalently, the (unique) element with $0$ in the first entry. Consequently, Proposition~\ref{prop:hilbert_is_norm_general} gives the representation of the metric (in $\theta_0$-coordinates)
		\[\hilbert(\mu,\nu) = \Big(\max_{i\neq 0}\big\{\theta_0^i(\mu) - \theta_0^i(\nu)\big\}\Big)^+ + \Big(\min_{i\neq 0}\big\{\theta_0^i(\mu) - \theta_0^i(\nu)\big\}\Big)^-,\]
		with $x^+ = \max\{0,x\}$ and $x^- = \max\{0, -x\}$. As this representation shows the Hilbert metric is given by a norm in $\theta_k$-coordinates (for any $k$), we know that translation in $\theta_k$-coordinates will not change the size or shape of a ball.
	\end{rmk}
	
	The $\theta_k$-coordinates allow us to investigate in detail the shape of Hilbert balls in $\mathcal{S}^n$. The main idea is as follows: let $\mu,\nu \in \mathring{\mathcal{S}}^n$, so $\mu$ and $\nu$ are equivalent, and for all $k \in \mathbf{N}$ consider the transformations $\mu \mapsto \theta_k(\mu)$ and $\nu \mapsto \theta_k(\nu)$. Let $\mathcal{H}(\mu, \nu) = R < \infty$. Fixing $\nu$ and $k=0$, we prove that the $\hilbert$-ball of radius $R$ around $\theta_{0}(\nu)$ is a convex polytope $\mathcal C \subset \R^n$. Mapping $\mathcal C$ to the simplex $\mathcal{S}^n$, we find that the image of $\mathcal C$ through the inverse transformation $\theta_0^{-1}$ is also a convex polytope.
	
	We start with the following lemma.
	\begin{lemma}\label{lemma:polytope_Rn}
		Consider two probability measures $\mu, \nu \in \mathring{\mathcal S}^n$ such that $\mathcal{H}(\mu,\nu) = R > 0$. Let $\theta_0(\mu), \theta_0(\nu) \in \R^n$ be their natural parameters under the mapping $\theta_0$ given by \eqref{eq:theta_transf}. Then $\theta_0(\mu)$ belongs to the boundary $\partial \mathcal C$ of an $n$-dimensional convex polytope $\mathcal{C} \subset \R^n$ centred at $\theta_0(\nu)$, with $2(2^n -1)$ vertices at the points
		\begin{equation}\label{eq:vertices_of_C}
			v_{\mathcal I}^+ := \theta_0(\nu) + R \sum_{i \in \mathcal I} e_{i}, \quad
			v_{\mathcal I}^- := \theta_0(\nu) - R \sum_{i \in \mathcal I} e_{i},
		\end{equation}
		where $\{e_i\}_{i=1}^n$ denote the basis vectors of $\R^n$ and $\mathcal I \subseteq \{1, 2, \dots, n \}$, $\mathcal I \neq \emptyset$.
	\end{lemma}
	
	\begin{proof}
		The first thing we do, to simplify our calculations, is to translate $\theta_0(\nu) \in \R^n$ to the origin $\mathbf 0$. Now let $\theta_k^i := \log (\mu^i / \mu^k) \in \R$ for all $i,k = 0, \dots, n$. We fix the coordinate system in $\R^n$ to be given by $(x^1, \dots, x^n) \equiv (\theta_0^1, \dots, \theta_0^n)$, so that the basis vectors $e_i$ are the unit vectors in the $\theta_0^i$-direction. We look for a representation of the $\hilbert$-ball of radius $R$ around the origin in this coordinate system.
		
		By \eqref{eq:theta_inf_and_hilbert}, clearly $| \theta_k^i| \le R$ for all pairs $(i,k) \in \mathbf{N} \times \mathbf{N}$. We  consider all these inequalities, noting that, since $| \theta_k^i| = | \theta_i^k|$ by properties of $\log$, we can avoid needless repetitions by restricting our consideration to all pairs of indices $(i,k) \in I^n : = \{ (i,k) \in \mathbf{N} \times \mathbf{N}  \, : \, i > k \}$. Then we have in total $n(n+1)/2$ unique inequalities. Recalling that $\theta^i_k = \theta^i_0 - \theta^k_0$, for $(i,k) \in I^n$ we define the $(n-1)$-dimensional hyperplanes 
		\begin{equation}\label{eq:hyperplanes}
			\mathdutchcal{h}_{i,k}^{\pm} := 
			\left\{
			\begin{array}{ll}
				\{ (x^1, \dots, x^n) \in \R^n \, : \, x^i - x^k = \pm R  \} &\quad k \neq 0, \\
				\{ (x^1, \dots, x^n) \in \R^n \, : \, x^i = \pm R  \} &\quad k = 0,
			\end{array}
			\right.
		\end{equation}
		and denote by $\mathdutchcal p^+_{i,k} := \{ x \in \R^n \, : \, x^i - x^k \le R  \}$ the half-spaces bounded by $\mathdutchcal{h}_{i,k}^{+}$, and similarly by $\mathdutchcal p^-_{i,k} := \{ x \in \R^n \, : \, x^i - x^k \ge - R  \}$ those bounded by $\mathdutchcal{h}_{i,k}^{-}$ (and equivalently when $k = 0$).
		
		We let
		$
		\mathcal C^n = \bigcap_{(i,k) \in I^n} \mathdutchcal p^+_{i,k} \cap \mathdutchcal p^-_{i,k}.
		$
		By standard results in $n$-dimensional geometry, $\mathcal C^n \subset \R^n$ is a polyhedron, since it is the intersection of a finite number of closed half-spaces. We claim $\mathcal C^n$ is bounded.
		
		Note that the intersection
		$
		\mathcal C_0^n = \bigcap_{i = 1}^n \mathdutchcal p^+_{i,0} \cap \mathdutchcal p^-_{i,0}
		$
		is the $n$-cube with side-length $2R$ centred at $\mathbf 0$ with $2^n$ vertices at all possible positive/negative combinations of the coordinates $( \pm R, \dots, \pm R)$. Then
		\begin{equation}\label{eq:intersection_cube_half_spaces}
			\mathcal C^n = \bigcap_{\substack{(i,k) \in I^n \\ k \neq 0}} \mathdutchcal p^+_{i,k} \cap \mathdutchcal p^-_{i,k} \cap \mathcal C_0^n,
		\end{equation}
		and the intersection of a hypercube with closed half-spaces is bounded, so $\mathcal C^n$ is a bounded polyhedron, and therefore a convex polytope. (Note that $\mathcal C^n$ is non-empty, since one can easily check that $\mathbf 0 \in \mathcal C^n$.)
		
		We now would like to find the vertices of $\mathcal C^n$. We look for all the points in $\R^n$ where exactly $n$ of the hyperplanes \eqref{eq:hyperplanes} intersect uniquely.
		
		Consider a linear system $S$ given by $n$ equations from \eqref{eq:hyperplanes}. We say indices $i,j$ are \textit{linked} if an equation of the form $x^i - x^j = \pm R$ appears in the system $S$, and extend this definition by transitivity to partition the indices appearing in $S$ into \textit{linked classes}. We say an index $i$ is a \textit{base case} if $x_i = \pm R$ appears in $S$.
		\begin{enumerate}
			\item Consider a class not containing a base case. Then we can add a constant $r \in \R$ to each component $x^i$ in the class without altering the equations of the form $x^i - x^j = \pm R$. Therefore the subsystem of $S$ containing all the equations for this class cannot have a unique solution, so $S$ cannot give a vertex.
			\item For any class containing a base case, let's say $x^i$, note that $(x^i - x^j)/R \in \Z$ for any $x^j$ linked to $x^i$. By transitivity and additive closure of $\Z$, we observe that all indices in a class containing a base case must have $x^j/R \in \Z$.
		\end{enumerate}
		By combining the above observations, all vertices of $\mathcal{C}^n$ must have coordinates that are integer multiples of $R$, with at least one coordinate given by a base-case, i.e. any point $v \in \R^n$ that solves uniquely $S$ must be of the form $(m_1 R, \dots, m_n R)$ for $m_i \in \Z$ with $-n \le m_i \le n$ for all $i = 1, \dots, n$, and at least one $m_i \in \{\pm 1\}$.
		
		Now, by \eqref{eq:intersection_cube_half_spaces} we must have that all the vertices of $\mathcal{C}^n$ belong to $\mathcal{C}^n_0$. Thus, any point $v \in \R^n$ that uniquely solves $S$ and is potentially a vertex of $\mathcal C^n$ must be of the form $(m_1 R, \dots, m_n R)$ with $m_i \in \{-1, 0, 1\}$ for all $i = 1, \dots, n$. Moreover, assume that $m_i = 1$ and $m_k = -1$ for $i > k$. Then $m_i R - m_k R = 2R \ge R$, so $v \notin \mathdutchcal p_{i,k}^+$ and $v \notin \mathcal C^n$. Similarly, if $m_i = -1$ and $m_k = 1$ for $i > k$, then $m_i R - m_k R \le -2R$ so $v \notin p_{i,k}^-$ and $v \notin \mathcal C^n$.
		
		Thus we must have that any vertex of $\mathcal C^n$ is of the form $(m_1 R, \dots, m_n R)$ with either $m_i \in \{-1, 0\}$ for all $i = 1, \dots, n$ or $m_i \in \{ 0, 1\}$ for all $i = 1, \dots, n$. Conversely, consider any point $v \in \R^n$ of this form. It is easy to construct a system $S$ for a choice of $n$ equations in $\{x^i - x^k = \pm R\} \cup \{x^j = \pm R\}$ such that $v$ solves $S$. Then the points $(m_1 R, \dots, m_n R) \in \R^n$ with either $m_i \in \{-1, 0\}$ for all $i = 1, \dots, n$ or $m_i \in \{ 0, 1\}$ for all $i = 1, \dots, n$ (and not all $m_i$ identically 0) are in fact all the vertices of $\mathcal C^n$.
		Letting $\mathcal C = \mathcal C^n + \theta_0(\nu)$, we are done.
	\end{proof}
	\begin{rmk}
		As can be deduced by \eqref{eq:theta_inf_and_hilbert}, the $\hilbert$-ball is, in a way, nothing but the intersection of $n+1$ skewed $\ell_{\infty}$-balls, or, geometrically speaking, the intersection of $n+1$ skewed hypercubes. Recall the notation from our proof above. Define the following intersections
		\begin{equation*}
			\mathcal C_k^n := \Big[ \bigcap_{j = k+1}^n \mathdutchcal p^+_{j,k} \cap \mathdutchcal p^-_{j,k} \Big] \cap \Big[ \bigcap_{j = 0}^{k-1} \mathdutchcal p^+_{k,j} \cap \mathdutchcal p^-_{k,j} \Big], \quad \forall k=1, \dots, n.
		\end{equation*}
		Then for each $k = 1, \dots, n$, $\mathcal C_k^n$ is the image of the hypercube centred at $ \theta_k(\nu)$ with side-length $2R$ under the linear transformation $\theta_k \mapsto \theta_0$.
	\end{rmk}
	
	We now map our convex polytope from $\R^n$ to $S^n$ through the inverse transformation $\theta_0^{-1}$ given by \eqref{eq:theta_transf_inv}. Note that if $\theta_0^{-1}$ were an affine transformation, then Lemma \ref{lemma:polytope_Sn} stated below would be trivially true. However, as $\theta_0^{-1}$ is not affine, a bit more work is required to prove that convexity, linearity of the boundary, and intersections are preserved. We give a depiction of this result in Figure~\ref{fig:hilbert_ball_3d}.
	\begin{lemma}\label{lemma:polytope_Sn}
		The $\hilbert$-ball of radius $R$ around $\theta_0(\nu)$, given by the convex polytope $\mathcal C$ of Lemma 4.2, maps to a convex polytope $\mathcal D$ centred at $\nu$ in $\mathring{\mathcal{S}}^n$ under the inverse transformation $\theta_0^{-1}: \R^n \rightarrow \mathring{\mathcal{S}}^n$. The vertices of $\mathcal D$ are given by the images of the vertices of $\mathcal C$ under the same transformation.
	\end{lemma}
	\begin{proof}
		We use the same notation as in Lemma \ref{lemma:polytope_Rn}. Recall that we centred our $\mathcal H$-ball at $\hat \theta_0 = \theta_0(\nu)$, the image of $\nu$ under the mapping $\theta_0 : \mathcal S^n \rightarrow \R^n$. Similarly, we will fix the centre of $\mathcal D$, the representation of the $\mathcal H$-ball in $\mathcal S^n$, at $\nu$, and proceed as if $\nu$ were known.
		
		We start by noting that the bounds $|\theta^i_k - \hat \theta_k^i| \le R$, which correspond to the linear constraints \eqref{eq:hyperplanes}, are equivalent to linear constraints in $\mathcal{S}^n$. Recalling the notation of the proof of Lemma~\ref{lemma:polytope_Rn}, consider the half-spaces $\mathdutchcal{p}_{i,k}^{+}$ under the transformation $\theta_0^{-1}: \R^n \rightarrow \mathcal S^n$. We compute
		\begin{align*}
			\theta_0^{-1} (\mathdutchcal{p}_{i,k}^{+})
			&=  \{ \theta_0^{-1}(\theta_0) \in \mathcal S^n \, : \, \theta^i_0 - \theta^k_0 \le \hat \theta^i_0 - \hat \theta^k_0 + R  \} \\
			&= \Big\{ \mu \in \mathcal S^n \, : \, \log \frac{\mu^i}{\mu^k} \le \log \frac{\nu^i}{\nu^k} + R  \Big\} 
			= \Big\{ \mu \in \mathcal S^n \, : \, \mu^i - \mu^k \frac{\nu^i}{\nu^k} e^{R} \le 0  \Big\},
		\end{align*}
		where we have used bijectivity of $\theta_0^{-1}$ and the fact that $\exp$ is increasing. Similarly,
		\begin{equation*}
			\theta_0^{-1}(\mathdutchcal{p}_{i,k}^{-}) =  \Big\{ \mu \in \mathcal S^n \, : \, \mu^i - \mu^k \frac{\nu^i}{\nu^k} e^{-R} \ge 0  \Big\}.
		\end{equation*}
		Note that $\theta_0^{-1}(\mathdutchcal{p}_{i,k}^{+})$ and $\theta_0^{-1}(\mathdutchcal{p}_{i,k}^{-})$ are $(n-1)$-dimensional flat subspaces of $\mathcal S^n$. In particular, since $\mathcal S^n$ is a subset of an $n$-dimensional affine space $A \cong \R^n$, we see that $\theta_0^{-1}(\mathdutchcal{p}_{i,k}^{+})$ and $\theta_0^{-1}(\mathdutchcal{p}_{i,k}^{-})$ are closed half-spaces of $A$, bounded by the hyperplanes  $ \mathdutchcal{l}_{i,k}^{+} := \{ \mu \in A \, : \, \mu^i = \mu^k \, \frac{\nu^i}{\nu^k} \, e^{R}  \}$ and $ \mathdutchcal{l}_{i,k}^{-} := \{ \mu \in A \, : \, \mu^i = \mu^k \, \frac{\nu^i}{\nu^k} \, e^{-R}  \}$, which are the images (extended to $A$) of respectively $\mathdutchcal{h}_{i,k}^{+}$ and $\mathdutchcal{h}_{i,k}^{-}$ under $\theta_0^{-1}$. Then, recalling that $I^n : = \{ (i,k) \in \mathbf{N} \times \mathbf{N}  \, : \, i > k \}$, the intersection
		\begin{equation*}
			\mathcal D : = \bigcap_{(i,k) \in I^n} \theta_0^{-1}(\mathdutchcal p^+_{i,k}) \cap \theta_0^{-1}(\mathdutchcal p^-_{i,k})
		\end{equation*}
		is a convex polyhedron in $A$, and in particular, since $\theta_0^{-1}$ is a bijection from $\R^n$ into $\mathring{\mathcal S}^n$,
		\begin{equation*}
			\theta_0^{-1}(\mathcal C ) = \theta_0^{-1} \bigg( \bigcap_{(i,k) \in I^n} \mathdutchcal p^+_{i,k}\cap \mathdutchcal p^-_{i,k} \bigg) = \mathcal D \subset \mathcal S^n.
		\end{equation*}
		
		Finally, both boundedness of $\mathcal D$ in $\mathcal S^n$ (in the sense that $\mathcal D$ is bounded away from $\partial \mathcal S^n$), so that $\mathcal D$ is a convex polytope in $\mathcal S^n$, and the fact that vertices are preserved under the mapping follow easily from $\theta_0$ being an homeomorphism between $\mathring{\mathcal S}^n$ and $\R^n$.
	\end{proof}
	
	\begin{figure}[htp]
		\centering
		\includegraphics[width=0.8\textwidth]{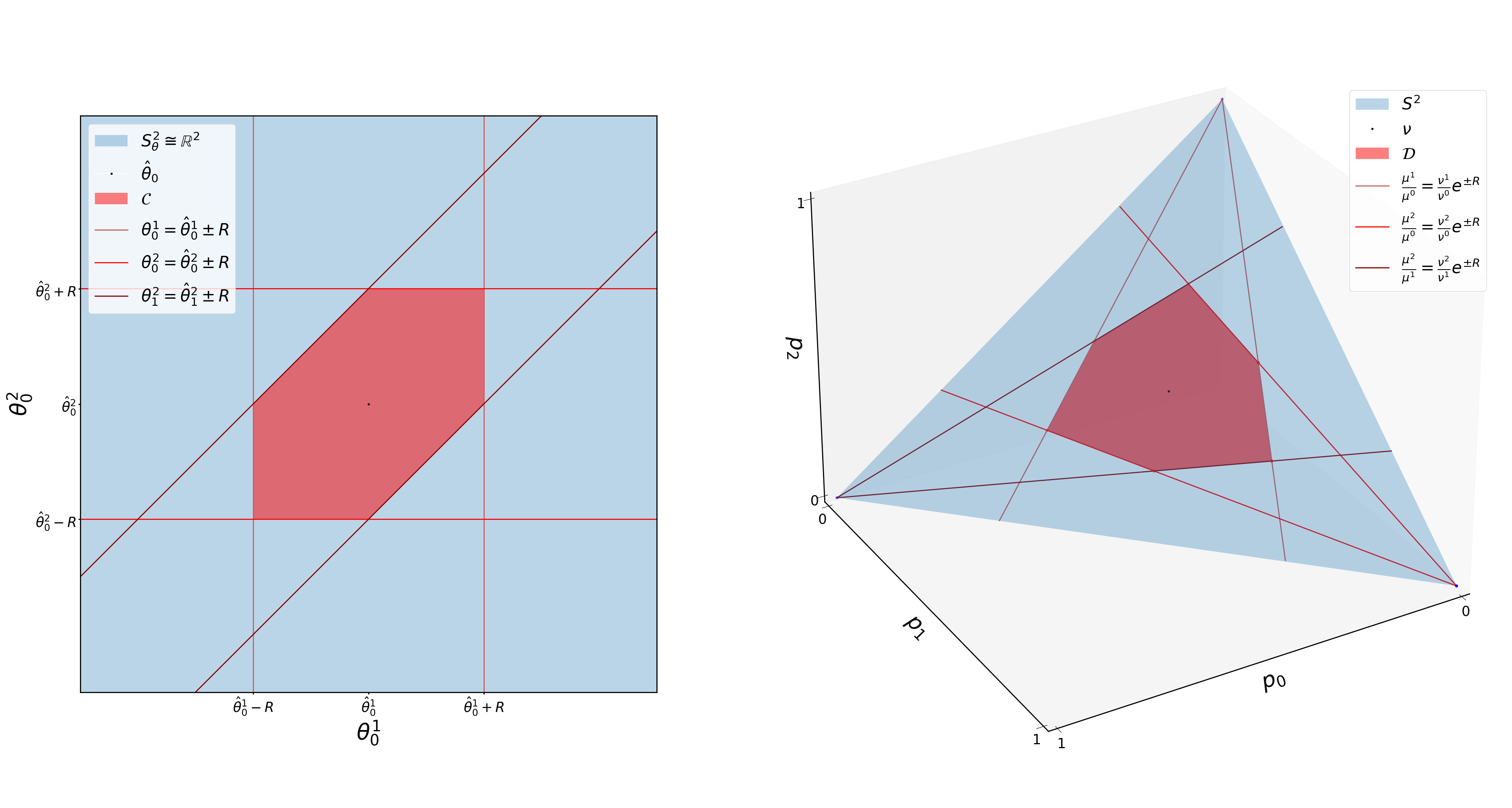}	
		\caption{On the left: representation in $\theta_0$-coordinates of a 2-dimensional $\mathcal{H}$-ball $\mathcal C$ of radius $R$ around $\hat \theta_0 = \theta_0(\nu) = (0,0)$. On the right, the image of $\mathcal C$ under $\theta_0^{-1}$, which gives the $\mathcal{H}$-ball $\mathcal D$ around $\nu = \big(\frac{1}{3}, \frac{1}{3}, \frac{1}{3} \big)$ as a hexagonal polygon in the simplex $\mathcal S^2$.}
		\label{fig:hilbert_ball_3d}
	\end{figure}
	
	\begin{rmk}
		Using Figure~\ref{fig:hilbert_ball_3d}, we can build an intuition of how the transformation $\theta_0^{-1}$ deforms $\mathcal C$ by considering what happens to the parallel pairs of hyperplanes $\mathdutchcal{h}_{i,k}^{+}$ and $\mathdutchcal{h}_{i,k}^{-}$ when mapped into $\mathcal S^n$.
		For each pair $(i,k) \in \mathbf{N} \times \mathbf{N}$, $\mathdutchcal{l}_{i,k}^{+} = \theta_0^{-1} (\mathdutchcal{h}_{i,k}^{+})$ and $\mathdutchcal{l}_{i,k}^{-} = \theta_0^{-1}(\mathdutchcal{h}_{i,k}^{-})$ are not parallel in $\mathcal S^n$, but meet at the $(n-2)$-face of the simplex given by $\mathdutchcal{f}_{i,k} = \{ \mu \in \mathcal S^n \, : \, \mu^i = 0, \mu^k = 0  \}$. In other words, the `point at infinity' at which $\mathdutchcal{h}_{i,k}^{+}$ and $\mathdutchcal{h}_{i,k}^{-}$ meet in $\R^n$ is mapped to the boundary of the simplex, and in particular to $\mathdutchcal{f}_{i,k}$, under $\theta_0^{-1}$.
		
		For example, in dimension 2, $\mathdutchcal{f}_{i,k}$ are vertices of $\mathcal S^2$: when mapping $\mathcal C \in \R^2$ to $\mathcal S^2$, we can think of squeezing together the $\infty$-extremities of each pair of parallel lines $\mathdutchcal{h}_{i,k}^{\pm}$ (for $(i,k) \in \{ (1,0), (2,0), (2,1)\}$) so that they meet at an angle of $\alpha = a(e^R - e^{-R})/(1+a^2)$, where $a = \nu^i/\nu^k$. Then we place the intersection point at the vertex $f_{i,k}$ of $\mathcal S^2$, so that, intuitively, the strip of plane between $\mathdutchcal h_{i,k}^+$ and $\mathdutchcal h_{i,k}^-$ is mapped to a slice of $\mathcal S^2$ of width $\alpha$ bounded by $l_{i,k}^+$ and $l_{i,k}^-$. Then it is easy to visualize how the straight lines that compose the boundary of $\mathcal C$ are mapped to straight lines, and how intersections are preserved, making $\theta_0^{-1} (\mathcal C)$ into a polytope as well. However, these lines (and those parallel to them) are in fact the only straight lines in $\theta_0$-coordinates that map to straight lines in $\mathcal{S}^2$ (as expected, since $\theta_0$ and its inverse are nonlinear). We illustrate this in Figure \ref{fig:lines_deformation} below.	
		
	\end{rmk}
	
	\begin{figure}[htp]
		\centering
		\includegraphics[width=0.7\textwidth]{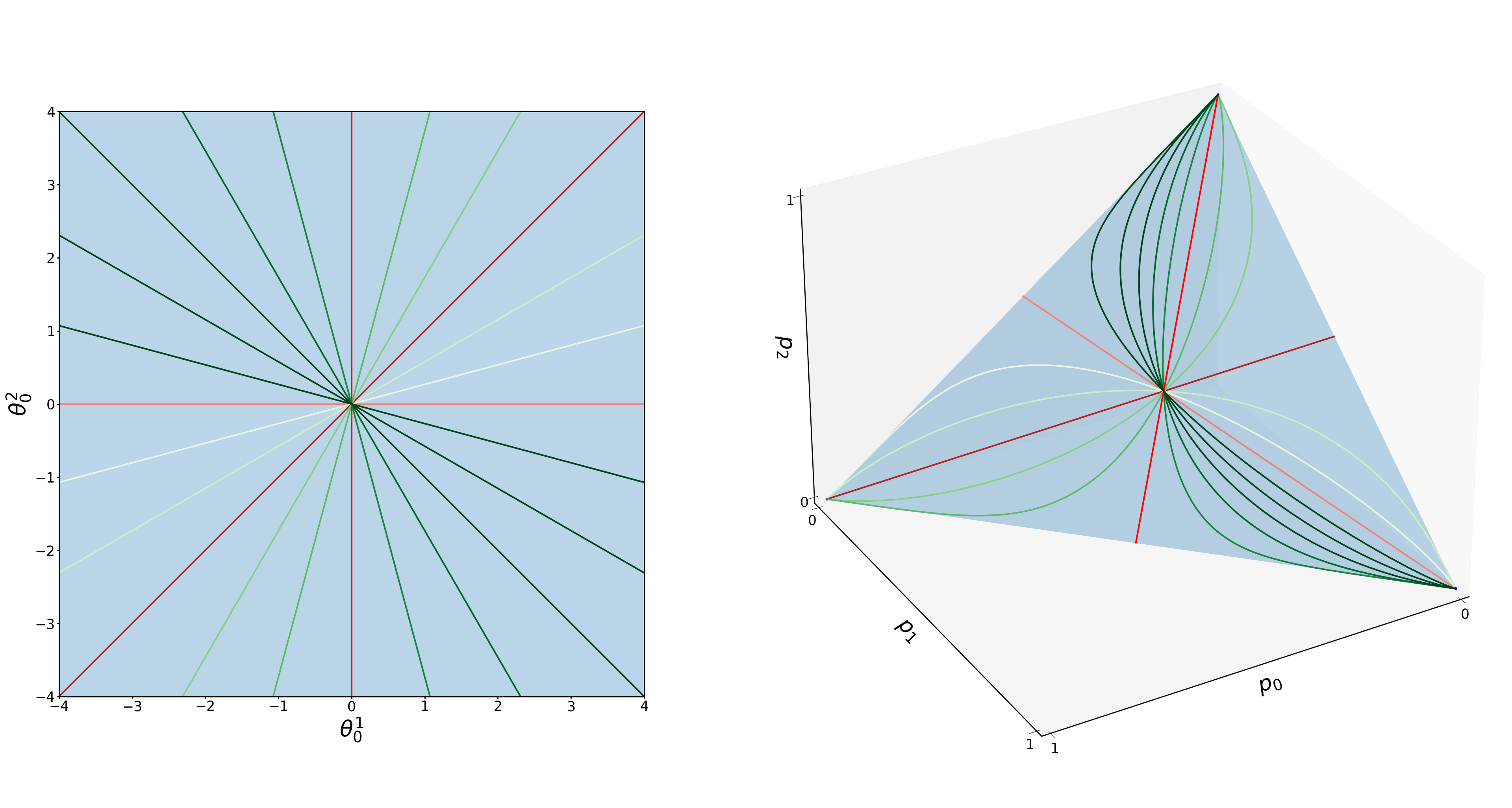}	
		\caption{Straight lines through the origin in $\theta_0$-coordinates on the left, and their images in $\mathcal{S}^2$ under the inverse mapping $\theta_0^{-1}$ on the right. In red the lines parallel to the axes and the diagonal in $\R^2$, which remain straight in $\mathcal{S}^2$.}
		\label{fig:lines_deformation}
	\end{figure}
	
	\begin{rmk}The regularity of the $\hilbert$-balls, when represented in $\theta_0$-coordinates, has other surprising consequences\footnote{In a more artistic vein, Figure \ref{fig:tiling_of_simplex} also illustrates that the map $\theta_0:S^n\to \mathbb{R}^n$, for $n=2$, corresponds to the classical transformation between parallel oblique perspective ($\theta_0$-coordinates) and three-point perspective (by viewing $S^n$ with its vertices at the three vanishing points), linking back well beyond Birkhoff (1957) and Hilbert (1895), at least as far as the work of Jean Pelerin (Viator) in \emph{De Artificiali Perspectiva} (1505).}---for example, Hilbert balls of constant radius naturally tile the space, as illustrated in Figure \ref{fig:tiling_of_simplex}. 
	\end{rmk}
	\begin{figure}[htp]
		\centering
		\includegraphics[width=0.8\textwidth]{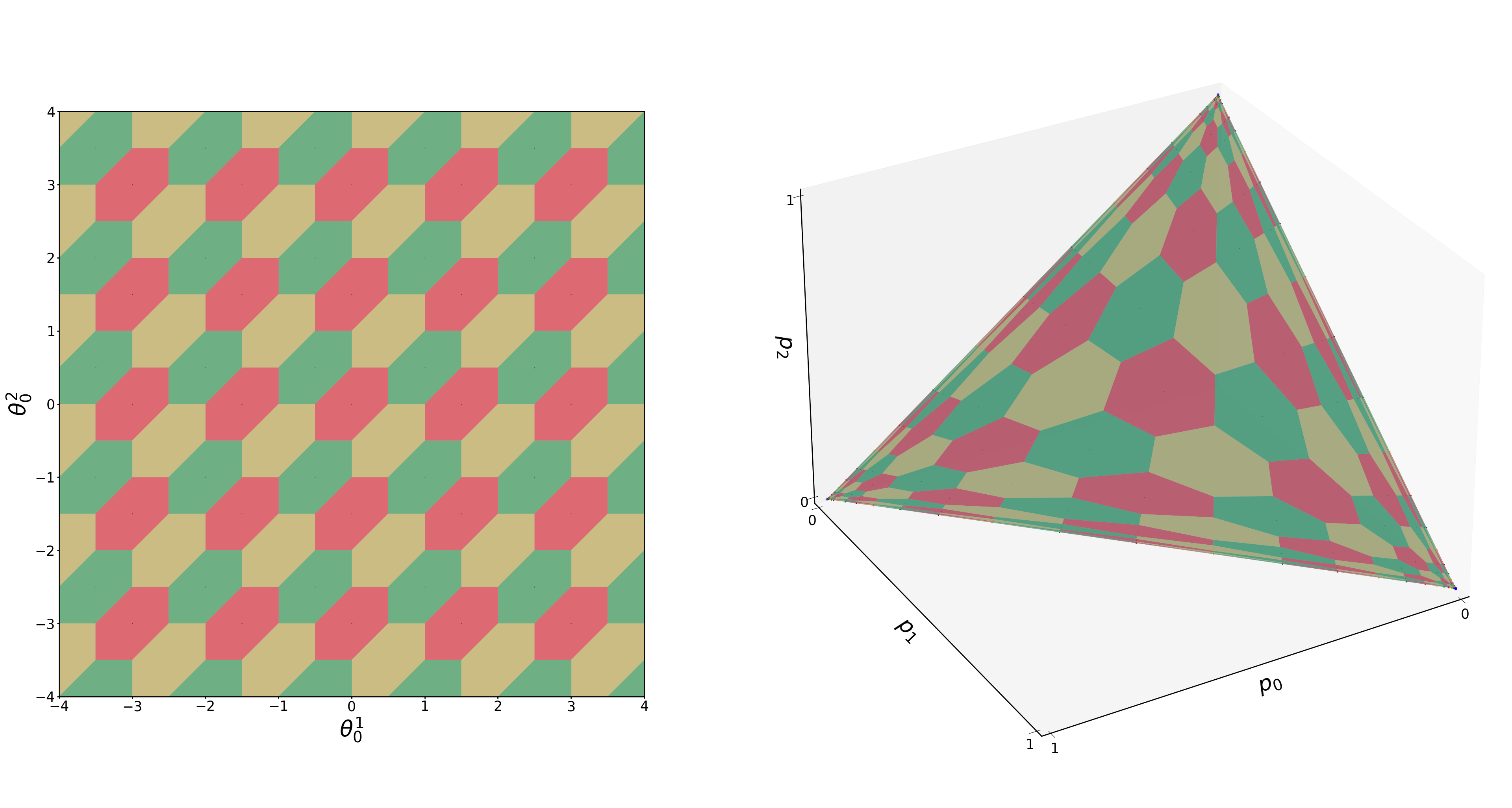}	
		\caption{Tiling of the $2$-dimensional probability simplex $\mathcal{S}^2$ with Hilbert balls of radius $0.5$, starting from the ball around the center $\big(\frac{1}{3}, \frac{1}{3}, \frac{1}{3}\big)$ (corresponding to the origin in the $\theta_0$-coordinates on the left).}
		\label{fig:tiling_of_simplex}
	\end{figure}
	
	\section{Metric-comparisons for probability measures: $\tvnormbig$ and $\hilbert$}
	
	We now exploit the geometric intuition we gathered in the previous section to derive a bound, sharper than \eqref{eq:atar_zeitouni_bound}, for the total variation norm with respect to the Hilbert projective metric. We start by working with discrete probabilities in the simplex $\mathcal{S}^n$ and then extend our result to probability measures on a general measurable space. Perhaps unsurprisingly, the distance $\mathcal{T}(\mu,\nu) = \tanh (\hilbert(\mu,\nu) / 4)$, which we defined in Section~\ref{sec:contraction}, plays a role once more in the computations below.
	
	\subsection{Probabilities on finite state-space}
	
	First of all, recall that for a measurable space $(E, \mathcal{F})$ the total variation distance \eqref{eq:total_var_norm_bm} between two probability measures $\mu, \nu \in \mathcal{P}(E)$ is equivalent to
	\begin{equation}\label{eq:total_var_indicator}
		\|\mu - \nu \|_{\tvnorm} = 2 \sup_{A \in \mathcal{\mathcal{F}}} | \mu(A) - \nu(A)|,
	\end{equation}
	which, in the case of $E \cong \{0,\cdots n\}$ and $\mu,\nu \in \mathcal{S}^n$, reduces to
	\begin{equation}\label{eq:tv_equal_l_infty}
		\|\mu - \nu \|_{\tvnorm} = \sum_{i = 0}^n |\mu^i - \nu^i | = \| \mu - \nu \|_{\ell^1}.
	\end{equation}
	\begin{rmk}
		Note that the factor of 2 in \eqref{eq:total_var_indicator} is usually dropped, in which case \eqref{eq:tv_equal_l_infty} would be stated as $\|\mu - \nu \|_{\tvnorm} = \tfrac{1}{2}\| \mu - \nu \|_{\ell^1}$. We keep the factor of 2 in analogy with Atar and Zeitouni \cite{atar-zeitouni97}.
	\end{rmk}
	
	\begin{thm}\label{thm:hilbert_and_L1_simplex}
		Given two probability measures $\mu, \nu \in \mathcal{S}^n$, we have that
		\begin{equation}\label{eq:hilbert_and_L1}
			\|\mu - \nu \|_{\tvnorm}=\norm{\mu - \nu}_{\ell^1} \le 2 \tanh \frac{\mathcal{H}(\mu,\nu)}{4}   .
		\end{equation}
		Equivalently, we have
		\begin{equation}\label{eq:hilbert_and_TV}
			\sup_{A \subseteq \{0, \dots, n\}} \sum_{i \in A} |\mu^i - \nu^i| \le  \mathcal{T}(\mu,\nu) .
		\end{equation}
	\end{thm}
	
	From Lemma~\ref{lemma:polytope_Sn} we know that if $\mu, \nu \in \mathring{\mathcal{S}}^n$ with $\mathcal{H} (\mu, \nu) = R < \infty$, then $\mu$ belongs to the boundary of a convex polytope $\mathcal D \in \mathcal S^n$, centred at $\nu$ and with vertices $\{ \theta_0^{-1} (v_{\mathcal I}^{+}), \theta_0^{-1} (v_{\mathcal I}^{-}) \, : \, \mathcal{I} \subseteq \{ 1, \dots, n\}, \, \mathcal{I} \neq \emptyset  \}$. Finding an upper bound for $\| \mu - \nu   \|_{\ell^1}$ is now a simple convex optimization problem: we know the $\ell^1$-distance between $\nu$  and $\mu$ is maximized when $\mu$ is at one of the vertices of $\mathcal D$, so we compute
	the $\ell^1$-distance between $\nu$ and each of these vertices, and then maximize over the choice of vertex.
	
	\begin{lemma}\label{lemma:L1_hilbert_pathwise}
		Assume $\nu \in \mathring{\mathcal S}^n$ is known. If $\mathcal H(\mu,\nu) = R$, the $\ell^1$-distance between $\mu$ and $\nu$ is bounded by
		\begin{equation}\label{eq:L1_hilbert_pathwise}
			\norm{\mu - \nu}_{\ell^1} \le 2 \max_{\substack{\mathcal I \subseteq \{1, \dots, n \} \\ \mathcal I \neq \emptyset} } \left\{
			\frac{  \mathbf{S}_{\mathcal I} ( 1 - \mathbf{S}_{\mathcal I} ) ( e^{R} - 1   )}{1+
				\mathbf{S}_{\mathcal I} ( e^{R} - 1)} \:\: \vee  \:\:
			\frac{  \mathbf{S}_{\mathcal I} ( 1 - \mathbf{S}_{\mathcal I} ) ( 1- e^{-R} )}{ 1+
				\mathbf{S}_{\mathcal I} ( e^{-R} - 1)} \right\},
		\end{equation}
		where $\mathbf{S}_{\mathcal I} : = \sum_{i \in \mathcal I} \nu^i$ for any subset $ \mathcal I$ (not $\emptyset$) of the indices.
	\end{lemma}
	\begin{proof}
		Recall \eqref{eq:vertices_of_C}. Consider a vertex $\mathbf{v}_{\mathcal I +} = \theta_0^{-1}(v_{\mathcal I}^+)$ of $\mathcal D \in \mathcal S^n$, and let $\mathbf{N}_0 = \mathbf{N} \setminus \{ 0 \}$. Compute
		\begin{align*}
			\| \nu - \mathbf{v}_{\mathcal I +}    \|_{\ell^1} &= \sum_{i = 0}^{n} | \nu^i -  \mathbf{v}_{\mathcal I +}^i  | \\
			&= \Bigg| \nu^0 -\frac{1}{1 + \sum_{i \in \mathcal I} \exp\{\theta_{0}^i(\nu) + R\} + \sum_{i \in \mathbf{N}_0 \setminus \mathcal I} \exp\{\theta_{0}^i(\nu)\} } \Bigg| \\
			&\quad + \sum_{k \in \mathcal I} \Bigg| \nu^k -\frac{\exp\{ \theta_0^k(\nu) + R \}}{1 + \sum_{i \in \mathcal I} \exp\{\theta_{0}^i(\nu) + R\} + \sum_{i \in \mathbf{N}_0 \setminus \mathcal I} \exp\{\theta_{0}^i(\nu)\} } \Bigg| \\
			&\quad + \sum_{k \in \mathbf{N}_0 \setminus \mathcal I} \Bigg| \nu^k -\frac{\exp\{ \theta_0^k(\nu) \}}{1 + \sum_{i \in \mathcal I} \exp\{\theta_{0}^i(\nu) + R\} + \sum_{i \in \mathbf{N}_0 \setminus \mathcal I} \exp\{\theta_{0}^i(\nu)\} } \Bigg| \\
			&= \Bigg| \nu^0 -\frac{1}{1 + e^R \sum_{i \in \mathcal I} \frac{\nu^i}{\nu^0} + \sum_{i \in \mathbf{N}_0 \setminus \mathcal I} \frac{\nu^i}{\nu^0} } \Bigg|
			+ \sum_{k \in \mathcal I} \Bigg| \nu^k -\frac{e^R \frac{\nu^k}{\nu^0}}{1 + e^R \sum_{i \in \mathcal I} \frac{\nu^i}{\nu^0} + \sum_{i \in \mathbf{N}_0 \setminus \mathcal I} \frac{\nu^i}{\nu^0} } \Bigg| \\
			&\quad + \sum_{k \in \mathbf{N}_0 \setminus \mathcal I} \Bigg| \nu^k -\frac{ \frac{\nu^k}{\nu^0}}{1 + e^R \sum_{i \in \mathcal I} \frac{\nu^i}{\nu^0} + \sum_{i \in \mathbf{N}_0 \setminus \mathcal I} \frac{\nu^i}{\nu^0} } \Bigg|.
		\end{align*}
		Letting $\mathbf{S}_{\mathcal I} : = \sum_{i \in \mathcal I} \nu^i$ and $ \mathbf{ \breve S}_{\mathcal I} : = \sum_{i \in \mathbf{N}_0 \setminus  \mathcal I} \nu^i$, so that $\mathbf{S}_{\mathcal I} + \mathbf{ \breve S}_{\mathcal I} + \nu^0 = 1$, some algebra yields
		\begin{equation*}
			\| \nu - \mathbf{v}_{\mathcal I +}   \|_{\ell^1} = 2 (e^R - 1) \frac{\mathbf{S}_{\mathcal I} (1 - \mathbf{S}_{\mathcal I}) }{1 +  \mathbf{S}_{\mathcal I} \, e^R - \mathbf{S}_{\mathcal I}} = : g_R^{+}(\mathbf{S}_{\mathcal I})
		\end{equation*}
		Equivalently, if we let $\mathbf{v}_{\mathcal I -} = \theta_0^{-1}(v_{\mathcal I}^-)$, we obtain
		\begin{equation*}
			\| \nu - \mathbf{v}_{\mathcal I -}    \|_{\ell^1} = 2 (1 - e^{-R}) \frac{\mathbf{S}_{\mathcal I} (1 - \mathbf{S}_{\mathcal I}) }{1 +  \mathbf{S}_{\mathcal I} \, e^{-R} - \mathbf{S}_{\mathcal I}} = : g_R^{-}(\mathbf{S}_{\mathcal I})
		\end{equation*}
		Then the $\ell^1$-distance between $\mu$ and $\nu$ is bounded by the maximum between $\| \nu - \mathbf{v}_{\mathcal I +}    \|_{\ell^1}$ and $\| \nu - \mathbf{v}_{\mathcal I -}    \|_{\ell^1}$ over all choices of vertices, which yields the lemma.
	\end{proof}
	
	\begin{proof}[Proof of Theorem \ref{thm:hilbert_and_L1_simplex}]
		Let $\mu, \nu \in \mathring{\mathcal{S}}^n$ be such that $\hilbert (\mu,\nu)=R$. Recall the notation from the proof of Lemma~\ref{lemma:L1_hilbert_pathwise}. Note that $g_R^{+}(x)$ and $g_R^{-}(x)$ for $x \in [0,1]$ are symmetric around $x = \frac{1}{2}$. By standard calculus, we find that the maximum of $g_R^{+}$ is attained at $x_{+}^{\ast} = \frac{1}{1+e^{R/2}} $; while $g_R^{-}$ is maximized at $x_{-}^{\ast} = 1 - x_{+}^{\ast} = \frac{e^{R/2}}{1+e^{R/2}}$. Evaluating $g_R^{+}$ and $g_R^{-}$ at their respective maximizers gives the upper bound
		\begin{align*}
			\| \nu - \mathbf{v}    \|_{\ell^1}
			&\le \max_{\substack{\mathcal I \subseteq \{1, \dots, n \} \\ \mathcal I \neq \emptyset} } \{ \| \nu - \mathbf{v}_{\mathcal I +}   \|_{L^1} \, \vee \, \| \nu - \mathbf{v}_{\mathcal I -}    \|_{L^1}   \} \\
			&\le \max_{\nu \in \mathring{\mathcal{S}}^n}\max_{\substack{\mathcal I \subseteq \{1, \dots, n \} \\  \mathcal I \neq \emptyset} } \{ \| \nu - \mathbf{v}_{\mathcal I +}   \|_{L^1} \, \vee \, \| \nu - \mathbf{v}_{\mathcal I -}    \|_{L^1}   \} \\
			&\le \max_{\mathbf{S}_{\mathcal I} \in [0,1]} \{ g_R^{+}(\mathbf{S}_{\mathcal I}) \, \vee \, g_R^{-}(\mathbf{S}_{\mathcal I})    \} \\
			&\le 2 \tanh \frac{R}{4} .
		\end{align*}
		Finally, note that the statement is trivial if $\mu \nsim \nu$. Moreover, if $\mu \sim \nu$ and $\mu, \nu \in \partial \mathcal{S}^n$, then they must belong to the same $(n-d)$-face of $\mathcal{S}^n$, which is also a probability simplex $\mathcal{S}^d$ with $1 \le d < n$. In particular, $\mu, \nu \in \mathring{\mathcal{S}}^d$, which is the same as the case we considered originally, so we are done.
	\end{proof}
	
	\begin{rmk}
		For low dimensions, Lemma \ref{lemma:L1_hilbert_pathwise} provides a way to compute a bound for $\|\mu - \nu \|_{\ell^1}$ which is tighter than Theorem~\ref{thm:hilbert_and_L1_simplex}. While this still holds true in higher dimensions, the improvement gained by computing explicitly the bound \eqref{eq:L1_hilbert_pathwise} instead of using \eqref{eq:hilbert_and_L1} can be negligible, since it is likely that at least one of the possible combinations for $\mathbf{S}_{\mathcal I}$ comes very close to the maximizer.
	\end{rmk}
	
	\subsection{Probabilities on a general measurable space}
	
	We now use the discrete result of Theorem~\ref{thm:hilbert_and_L1_simplex} to give bounds for general probability measures.
	
	\begin{corollary}\label{cor:tv_H_bound_prob_measures}
		Let $(E, \mathcal{F})$ be a measurable space. Consider $\mu, \nu \in \mathcal{P}(E)$. We have
		\begin{equation*}
			\frac{1}{2}\|\mu - \nu \|_{\tvnorm} = \sup_{A \in \mathcal{F}} | \mu(A) - \nu(A)| \le \mathcal{T} (\mu, \nu).
		\end{equation*}
	\end{corollary}
	
	\begin{proof}
		Note that if $\mu \nsim \nu$, then the statement follows trivially, so assume that $\mu \sim \nu$. By definition of the total variation distance \eqref{eq:total_var_indicator}, for all $n \in \N$ there exists a set $A_n \in \mathcal{F}$ such that
		\begin{equation}\label{eq:tv_An_set}
			\frac{1}{2}\|\mu - \nu \|_{\tvnorm} - \frac{1}{n} \le | \mu(A_n) - \nu(A_n) | \le \frac{1}{2}\|\mu - \nu \|_{\tvnorm}.
		\end{equation}
		
		Let $\mathcal{F}_n$ be the $\sigma$-algebra generated by $A_n$, i.e. $\mathcal{F}_n = \{ A_n, A_n^{c}, E, \emptyset \}$, and let $\pi_n$ be the partition of $E$ given by $\pi_n = \{ A_n, A_n^c \}$. Consider the probability measures $\mu_n, \nu_n$ on the space $(E, \mathcal{F}_n)$ given by
		\begin{equation*}
			\mu_n = \mu |_{\pi_n}, \qquad \nu_n = \nu|_{\pi_n}.
		\end{equation*}
		Then $\mu_n$ and $\nu_n$ are probabilities on the finite state space $\{ A_n, A_n^c \}$, and in particular $\mu_n, \nu_n \in \mathcal{S}^1$. Therefore it holds that
		\begin{align*}
			\sup_{A \in \mathcal{F}_n} | \mu_n(A) - \nu_n(A) | 
			&= |\mu_n(A_n) - \nu_n(A_n)|  \\
			&\le \tanh \frac{\hilbert (\mu_n, \nu_n)}{4}
			= \tanh \frac{ \Big| \log \frac{\mu_n(A_n)}{\nu_n(A_n)} - \log \frac{\mu_n(A_n^c)}{\nu_n(A_n^c)} \Big|}{4}  \\
			&\le \tanh \frac{ \sup_{A,B \in \mathcal{F}}  \Big( \log \frac{\mu(A)}{\nu(A)} - \log \frac{\mu(B)}{\nu(B)} \Big)}{4}
			= \tanh \frac{\hilbert (\mu, \nu)}{4}.
		\end{align*}
		
		As for the left-hand side, we have that 
		\begin{equation*}
			\sup_{A \in \mathcal{F}_n} | \mu_n(A) - \nu_n(A) | 
			= |\mu_n(A_n) - \nu_n(A_n) |
			= | \mu(A_n) - \nu(A_n) |
			\le \sup_{A \in \mathcal{F}} | \mu(A) - \nu(A) |.
		\end{equation*}
		But by \eqref{eq:tv_An_set} we also have that
		\begin{equation*}
			\sup_{A \in \mathcal{F}_n} | \mu_n(A) - \nu_n(A) | 
			= |\mu_n(A_n) - \nu_n(A_n) |
			= | \mu(A_n) - \nu(A_n) |
			\ge \sup_{A \in \mathcal{F}} | \mu(A) - \nu(A) | - \frac{1}{n}.
		\end{equation*}
		Therefore, putting together the two above inequalities, we arrive at the final expression
		\begin{equation*}
			\lim_{n \to \infty} \sup_{A \in \mathcal{F}_n} | \mu_n(A) - \nu_n(A) | = \sup_{A \in \mathcal{F}} | \mu(A) - \nu(A) |,
		\end{equation*}
		and the result follows.
	\end{proof}
	
	\begin{rmk}
		Note that the bounds in Theorem~\ref{thm:hilbert_and_L1_simplex} and Corollary~\ref{cor:tv_H_bound_prob_measures} can be attained, and therefore are sharp. In particular, $2\mathcal{T}$ is nothing but the maximum $\tvnormbig$ (or $\ell^1$) norm between two probability measures which are a fixed $\hilbert$-distance apart. Theorem~\ref{thm:tanh_contaction} then tells us that this quantity contracts under (positive) linear transformations.
	\end{rmk}
	
	\begin{rmk}
		For $\mu,\nu \in \mathcal{P}(E)$, we can also find an upper bound for $\mathcal{T}(\mu,\nu)$ in terms of $\tvnormbig$. For all $A \in \mathcal{F}$, we have
		\begin{equation*}
			\tanh \bigg( \frac{1}{2} \log \frac{\mu(A)}{\nu(A)}   \bigg) 
			= \frac{\mu(A) - \nu(A)}{\mu(A) + \nu(A)},
		\end{equation*}
		therefore
		\begin{equation}
			\mathcal{T}(\mu,\nu)
			\le \tanh \bigg( \frac{1}{2} \log \Big( \sup_{A \in \mathcal{F}} \frac{\mu(A)}{\nu(A)} \lor  \sup_{B \in \mathcal{F}} \frac{\mu(B)}{\nu(B)} \Big)  \bigg) 
			\le \frac{\sup_{A \in \mathcal{F}} |\mu(A) - \nu(A)|}{2 \big(\inf_{B \in \mathcal{F}} \mu(B) \wedge \inf_{B \in \mathcal{F}} \nu(B)\big)}.
		\end{equation}
		However, we note that the denominator of the right-hand side may become arbitrarily small, which is unsurprising given $\mathcal{T}$ is a stronger metric than $\tvnormbig$ on $S^n$.
	\end{rmk}
	
	\paragraph{Acknowledgements} The research of EF was supported by the EPSRC under the award EP/L015811/1. SC acknowledges the support of the UKRI Prosperity Partnership Scheme (FAIR) under EPSRC Grant EP/V056883/1, the Alan Turing Institute and the Office for National Statistics (ONS), and the Oxford--Man Institute for Quantitative Finance.
	
	%%-------------------------------------------------------------------------------
	\newpage
	\appendix
	\section{Convergence rates of finite state-space Markov chains}\label{app:markov}
	
	In view of Proposition~\ref{prop:contraction_coeff}, we devote a brief appendix to the applications of the Hilbert metric and the Birkhoff's contraction coefficient to the study of the stability of discrete- and continuous-time Markov chains on finite state-spaces. We stress that this is by no means an exhaustive treatment of the topic, but, in the same spirit as the rest of the paper, it is meant to highlight the advantages and disadvantages of these tools in a probabilistic setting.
	
	Consider a homogeneous discrete-time Markov chain $(X_k)$ for $k= 0, 1, \dots$ on $(n+1)$-states, with transition matrix $P$ and $X_0 \sim \mu$, $\mu \in \mathcal{S}^n$. $P$ is a \textit{stochastic matrix} (it is non-negative and all rows sum to 1). Denote the state-space by $E\cong \{0,\cdots, n\}$. Theorem~\ref{thm:birkhoff} and Proposition~\ref{prop:contraction_coeff} immediately give that if $P$ is strictly positive, then the Markov chain forgets its initial distribution at a geometric rate, in the sense that, for $\nu \neq \mu$,
	\begin{equation}\label{eq:discrete_MC_hilbert}
		\hilbert ( \mu^\intercal P^{k}, \nu^\intercal P^{k}) \le \tau(P)^k \hilbert (\mu, \nu), \quad  \mathrm{for} \: k = 0, 1, \dots,
	\end{equation}
	with $\tau(P) < 1$. If $X_k$ is started from a Dirac measure, so $X_0 \sim \delta_{x}$ for $x \in E$, we can use the $\mathcal{T}$-distance to obtain a similar geometric contraction
	\begin{equation}\label{eq:discrete_MC_tanh}
		\mathcal{T}(\delta_{x}^\intercal P^k, \nu^\intercal P^k) \le \tau(P)^k \mathcal{T}(\delta_{x}, \nu) = \tau(P)^k, \quad \mathrm{for} \: k = 0, 1, \dots,
	\end{equation}
	and in general the $\mathcal{T}$-distance is always preferable when $\nu \nsim \mu$.
	
	We easily spot the main drawback of using the Birkhoff's contraction coefficient and the Hilbert distance to study the ergodicity of Markov chains: if $P$ is not strictly positive, then $\tau(P) = 1$, and neither \eqref{eq:discrete_MC_hilbert} nor \eqref{eq:discrete_MC_tanh} give us information about the convergence of $X_k$ to stability. It is known (see e.g.~\cite[Chapter~15]{meyn2009markov}), however, that irreducibility and aperiodicity of $X_k$ are enough to prove convergence of $X_k$ to equilibrium at a geometric rate. These two properties (irreducibility and aperiodicity) are equivalent to the transition matrix $P$ being \textit{primitive} (see e.g.~\cite[Thm.~1.4]{seneta_matrices}): a square non-negative matrix $P$ is called \textit{primitive} if there exists a positive integer $r$ such that $P^r$ is strictly positive. Thus, we can recover a contraction in Hilbert metric by considering the Markov chain $Y_k = X_{kr}$ instead, and obtain
	\begin{equation*}
		\mathcal{T} ( \mu^\intercal P^{kr}, \nu^\intercal P^{kr}) \le \tau(P^r)^k \mathcal{T} (\mu, \nu), \quad  \mathrm{for} \: k = 0, 1, \dots,
	\end{equation*}
	where $\tau(P^r) < 1$. More generally, note that for any two non-negative matrices $P$ and $Q$ we have \cite[Eq.~3.7]{seneta81}
	\begin{equation}\label{eq:tau_submultiplicativity}
		\tau(PQ) \le \tau(P) \tau(Q),
	\end{equation}
	which in particular implies that $\tau(P^s)$ is decreasing in $s \in \Z_+$. Then it is easy to show the following proposition.
	\begin{prop}\label{prop:discrete_primitive_mat}
		Let $P$ be a primitive stochastic matrix. For any $r \in \Z_+$, and $\mu,\nu \in \mathcal{S}^n$, there exists $C_r \in \R_+$ such that
		\begin{equation}\label{eq:r_root_bound}
			\mathcal{T} (\mu^\intercal P^k, \nu^\intercal P^k) \le C_r \big( \tau(P^r)^{\frac{1}{r}}\big)^k \mathcal{T} (\mu, \nu), \quad \forall k \in \Z_+,
		\end{equation}	
		where $C_r = \tau(P) \tau(P^r)^{-1} \ge 1$.
	\end{prop}
	\begin{proof}
		Note that when $r = 1$, the statement is just Theorem~\ref{thm:tanh_contaction}, so without loss of generality assume $r \ge 2$. By Theorem~\ref{thm:tanh_contaction} and submultiplicativity of $\tau(\cdot)$, we have
		\begin{equation*}
			\mathcal{T} (\mu^\intercal P^k, \nu^\intercal P^k) 
			\le \tau(P^k) \mathcal{T}(\mu,\nu)
			\le \tau(P^{r \lfloor \frac{k}{r} \rfloor}) \tau(P^{k - r\lfloor \frac{k}{r} \rfloor}) \mathcal{T} (\mu, \nu)
		\end{equation*}	
		where $\lfloor x \rfloor$ denotes the floor function. On one side, this yields the bound
		\begin{equation}\label{eq:floor_bound}
			\mathcal{T} (\mu^\intercal P^k, \nu^\intercal P^k) \le \tau(P^r)^{ \lfloor \frac{k}{r} \rfloor} \mathcal{T} (\mu, \nu), \quad \forall k \in \Z_+.
		\end{equation}
		On the other, since $\tau(P^s)$ is decreasing in $s$, we have
		\begin{equation*}
			\mathcal{T} (\mu^\intercal P^k, \nu^\intercal P^k) 
			\le \tau(P^r )^{\lfloor \frac{k}{r} \rfloor + 1} \frac{\tau(P^{k - r\lfloor \frac{k}{r} \rfloor})}{ \tau(P^r )} \mathcal{T} (\mu, \nu)
			\le \tau(P^r )^{\frac{k}{r}} C_r \mathcal{T} (\mu, \nu),
		\end{equation*}
		where $C_r =  \frac{\tau(P)}{ \tau(P^r )} \ge \frac{\tau(P^{k - r\lfloor \frac{k}{r} \rfloor})}{ \tau(P^r )} \ge 1$ for all $k$, since $1 \le k - r\lfloor \frac{k}{r} \rfloor < r$.
	\end{proof}
	
	\begin{rmk}
		To illustrate the result of Proposition \ref{prop:discrete_primitive_mat}, in Figure~\ref{fig:discrete_MC}(right)  we depict the convergence of a Markov chain $X$ on 7 states $\{ 0, 1, \dots, 6\}$ such that $X_k = i$ has probability $1/2$ of jumping to state $i+1$, and $1/2$ of jumping to state $i-1$ (modulo 7). Note that the associated transition matrix $P$ is primitive with $P^4 > 0$. By starting $X$ at state 0 and 6 respectively (so $\mu = \delta_0$ and $\nu = \delta_6$), and comparing the laws of the two chains using the $\mathcal{T}$-distance, we see that the worst-case bound \eqref{eq:floor_bound} is attained (see Figure~\ref{fig:discrete_MC}(left)).
	\end{rmk}
	
	\begin{figure}[htp]
		\centering
		\includegraphics[width=0.42\textwidth]{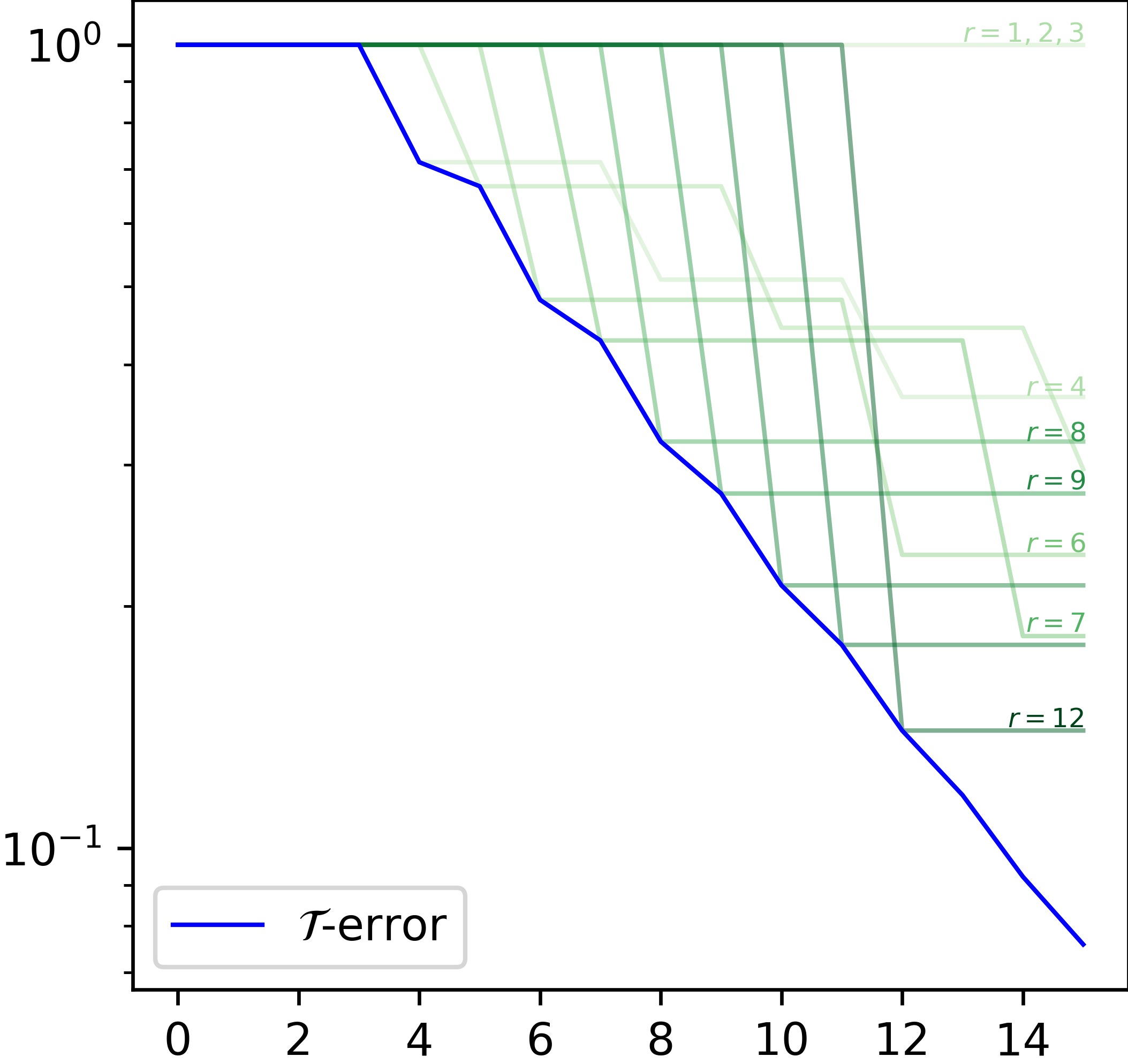}	
		\hspace{10pt}
		\includegraphics[width=0.42\textwidth]{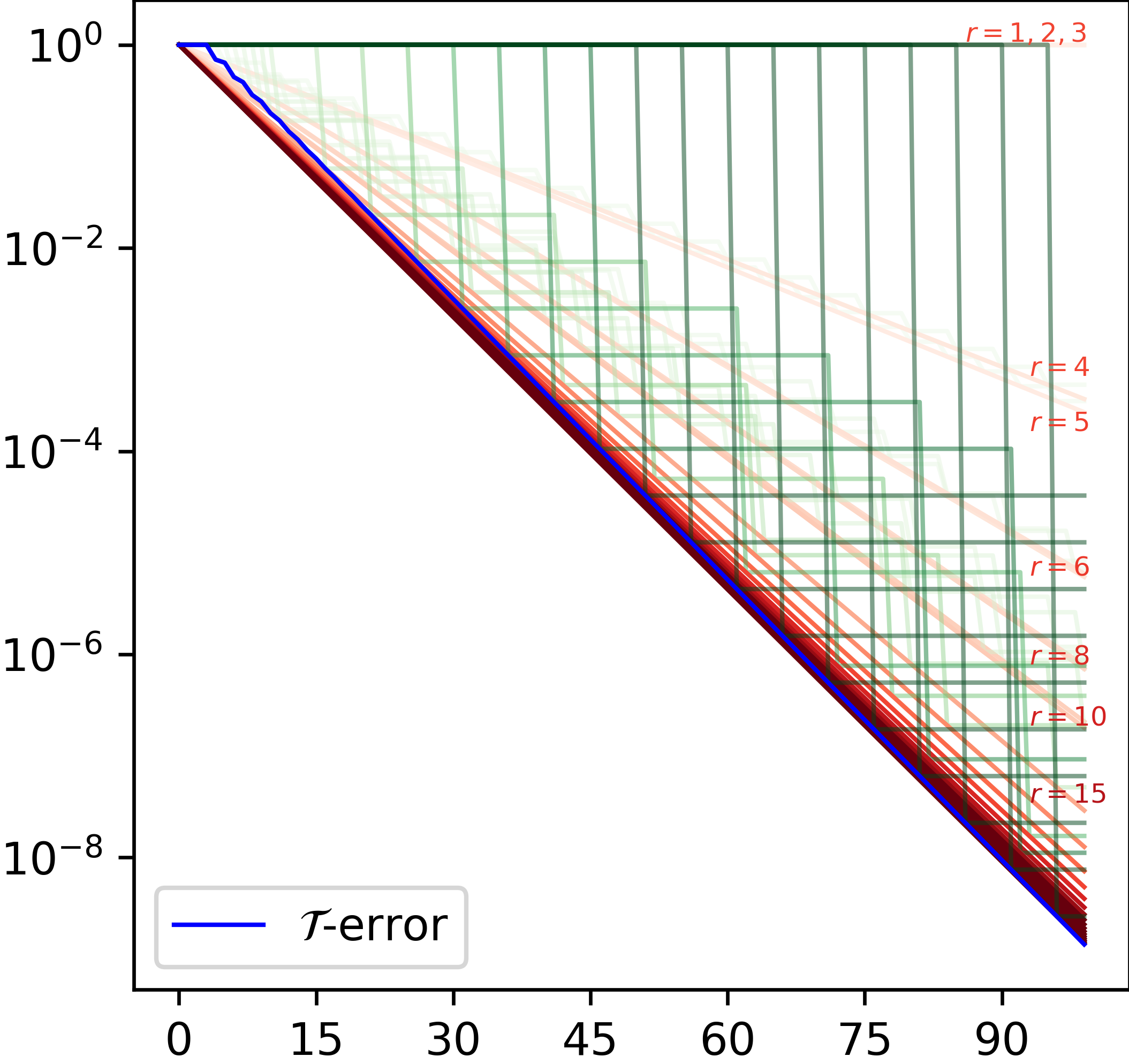}	
		\caption{Short-time (left) and long-time (right) behaviour of the $\mathcal{T}$-error between the distributions of the Markov chain specified in Remark~10 started from two different Dirac measures. The $y$-axis is in log-scale. We plot the $\mathcal{T}$-error in blue. On the left, we plot the bound \eqref{eq:floor_bound} for $r= 1, \dots, 12$ in green (progressively darker as $r$ increases). On the right, for $r = 1, \dots, 10, 15, \dots, 100$, we plot the bounds \eqref{eq:floor_bound} in green, and lines with slope $\frac{\log \tau(P^r)}{r}$, corresponding to the right-hand side of \eqref{eq:r_root_bound} fixing $C_r =1$, in red (progressively darker).}
		\label{fig:discrete_MC}
	\end{figure}
	
	Let us move to the continuous-time setting. Let $(X_t)_{t \ge 0}$ be a continuous-time, time-homogeneous Markov chain with finite state-space $E \approxeq \{0, \dots n\}$ and associated transition intensities matrix $Q = (Q_{ij}) \in \R^{(n+1)\times(n+1)}$. Recall that the $Q$-matrix is defined such that its entries for each row sum to 0, and its off-diagonal entries are non-negative, i.e.~$\sum_{j} Q_{ij} = 0$ for all $i$, and $Q_{ij} \geq 0$ for all $i,j \le n+1, \: i \neq j$.
	
	Let $P(t) = (P_{ij}(t)) = (\Prob (X_t = j | X_0 = i))$ be the transition function of $X$. Since $Q$ is the infinitesimal generator of $X$, we have that $P'(t) = QP(t)$, or, in other words, $P(t) = \exp (Qt)$. Then Theorem~\ref{thm:tanh_contaction} for continuous-time Markov chains reads as follows.
	\begin{prop}\label{prop:MC_cts_time_base}
		Let $(X_t)_{t \ge 0}$ be a continuous-time Markov chain on $E \approxeq \{0, \dots n\}$, with associated $Q$-matrix $Q = (Q_{ij}) \in \R^{(n+1) \times (n+1)}$ and transition matrix $P(t) = \exp ( Qt)$. Then for all $\mu, \nu \in \mathcal{S}^n$,
		\begin{equation}\label{eq:MC_cts_time_base}
			\mathcal{T} \big(\mu^\intercal P(t), \nu^\intercal P(t) \big) \le e^{- \lambda t} \mathcal{T} (\mu, \nu),
		\end{equation}
		where $\lambda = 2 \min_{i \neq j} \sqrt{Q_{ij}Q_{ji}}$.
	\end{prop}
	\begin{proof}
		Since $P(t) : \mathcal{S}^n \to \mathcal{S}^n$, we can apply Theorem~\ref{thm:tanh_contaction} to yield
		\begin{equation}\label{eq:MC_cts_base}
			\mathcal{T} \big(\mu^\intercal P(t), \nu^\intercal P(t)\big) \le \tau(P(t)) \mathcal{T} (\mu, \nu),
		\end{equation}
		where $\tau(P(t))$ is Birkhoff's contraction coefficient. For small $\varepsilon \ll 1$, consider $\tau(P(\varepsilon))$. Using \eqref{eq:bir_coeff_matrix} and expanding the matrix exponential, we compute
		\begin{equation*}
			\phi (P(\varepsilon)) = \phi (e^{Q \varepsilon}) 
			= \min_{i,j,k,l} \frac{(\delta_{ik} + Q_{ik} \varepsilon + O(\varepsilon^2))(\delta_{jl} + Q_{jl} \varepsilon + O(\varepsilon^2))}{(\delta_{il} + Q_{il} \varepsilon + O(\varepsilon^2))(\delta_{jk} + Q_{jk} \varepsilon + O(\varepsilon^2))},
		\end{equation*}
		and we see that for small enough $\varepsilon$ the minimum must be attained at some $i = l \neq j = k$, which yields
		\begin{equation*}
			\phi (P(\varepsilon))
			= \min_{i \neq j} \frac{ Q_{ij} Q_{ji} \varepsilon^2 + O(\varepsilon^3)}{1 + F_{ij}}, \quad \text{where} \: F_{ij} = (Q_{ii} + Q_{jj}) \varepsilon + O(\varepsilon^2).
		\end{equation*}
		Since the diagonal entries of $Q$ are non-positive, we have $-1 < F_{ij} \le 0$ for small enough $\varepsilon$, therefore there exists a $C > 0$ with $0 \le C\varepsilon < 1$ such that
		\begin{equation*}
			\min_{i \neq j} Q_{ij} Q_{ji} \varepsilon^2 + O(\varepsilon^3)
			\le
			\phi (P(\varepsilon))
			\le \min_{i \neq j} \frac{ Q_{ij} Q_{ji} \varepsilon^2}{1 - C \varepsilon} + O\Big(\tfrac{\varepsilon^3}{1 - C \varepsilon}\Big).
		\end{equation*}
		Let $\beta = \min_{i \neq j}  Q_{ij} Q_{ji} $. By \eqref{eq:bir_coeff_matrix} and using that $\frac{1-\sqrt{x}}{1+\sqrt{x}}$ is decreasing in $x$, we obtain
		\begin{equation*}
			\frac{ 1 - \sqrt{\beta \frac{ \varepsilon^2}{1 - C \varepsilon} + O\Big(\tfrac{\varepsilon^3}{1 - C \varepsilon}\Big)}}{1 + \sqrt{\beta \frac{ \varepsilon^2}{1 - C \varepsilon} + O\Big(\tfrac{\varepsilon^3}{1 - C \varepsilon}\Big)}}
			\le
			\tau (P(\varepsilon))
			\le 	\frac{1 - \sqrt{\beta \varepsilon^2 + O(\varepsilon^{3})}}{1 +  \sqrt{\beta \varepsilon^2 + O(\varepsilon^{3})}}.
		\end{equation*}
		Taking $\log$ and expanding it as $\log(z) = 2 \sum_{k = 0}^{\infty} \frac{1}{2k + 1} \Big( \frac{z -1}{z +1}\Big)^{2k+1}$ on the left- and right-hand sides, yields
		\begin{align*}
			-2 \Bigg[  \sqrt{\beta \frac{ \varepsilon^2}{1 - C \varepsilon} + O\Big(\tfrac{\varepsilon^3}{1 - C \varepsilon}\Big)} + O\Big(\tfrac{\varepsilon^3}{(1 - C \varepsilon)^{3/2}}\Big)
			\Bigg]
			\le \log \tau (P(\varepsilon))
			\le   - 2 \sqrt{\beta \varepsilon^2 + O(\varepsilon^{3})},
		\end{align*}
		so finally dividing by $\varepsilon$ and taking the limit (recalling that $\log \tau (P(0)) = \log \tau (\mathrm{Id}) = 0$) gives
		\begin{equation*}
			\frac{\di}{\dt} \log \tau(P(t)) \big|_{t=0} = \lim_{\varepsilon \to 0} \frac{1}{\varepsilon} \log \tau (P(\varepsilon)) = -2 \sqrt{\beta}.
		\end{equation*}
		Now recall \eqref{eq:tau_submultiplicativity}. Then
		\begin{equation*}
		\log \tau(P(t+s)) = \log \tau (e^{Q(t+s)})
		\le \log \Big( \tau(e^{Qt}) \tau(e^{Qs})   \Big) = \log \tau(P(t)) + \log \tau(P(s)),
		\end{equation*}
		and $\log \tau (P(t))$ is sub-additive. Moreover, for all $t \ge 0$ we have
		\begin{equation*}
			\frac{\di}{\dt} \log \tau(P(t)) = \lim_{\varepsilon \to 0} \frac{\log \tau (P(t+\varepsilon)) - \log \tau (P(t)) }{\varepsilon} \le \lim_{\varepsilon \to 0} \frac{\log \tau (P(\varepsilon))}{\varepsilon} = \frac{\di}{\dt} \log \tau(P(0)),
		\end{equation*}
		so the first derivative of $\log \tau(P(t))$ is non-increasing, which implies that the second derivative is non-positive. Then $\log \tau(P(t))$ is concave, and therefore it is bounded by $- 2 \sqrt {\beta} t$ for all $t \ge 0$. Exponentiating it, we get
		\begin{equation*}
			\tau(P(t)) \le e^{- 2 \sqrt{\beta} t},
		\end{equation*}
		and letting $\lambda = 2 \sqrt{\beta}$ and inserting this estimate in \eqref{eq:MC_cts_base}, we are done.
	\end{proof}

	Just as in the case of discrete-time Markov chains, we find ourselves in the situation
	that when $Q$ has zero-entries the coefficient $\lambda$ in \eqref{eq:MC_cts_time_base} is $0$ and Proposition \ref{prop:MC_cts_time_base} does not give us any information
	on the stability of $X$. 
	
	Nevertheless, inspired by our observations in Proposition \ref{prop:discrete_primitive_mat} for the
	discrete-time case, we can try to find a bound for $\mathcal{T}(\mu^\top P (t), \nu^\top P (t))$, which may not be a contraction, but still provides information on the decay of the error, given sufficient ergodicity conditions on $P(t)$ (and therefore, on $Q$).  
	
	We begin by deriving some algebraic properties of $Q$-matrices, based on (discrete) paths in the state space.
	\begin{defn}
		For a pair of states $i,j \in E$, a path from $i$ to $j$ is a sequence of states $S_{i \rightsquigarrow j} = (s_k)_{k=0}^m$ such that $s_0 = i$ and $s_m = j$. We say such a path is possible if $Q_{s_{a-1} s_{a}}\neq 0$ for all $a \in S_{i \rightsquigarrow j}$, and that it has no self-transitions if $s_{a} \neq s_{b}$ for all $a \neq b$. The length of a path, which we denote by $| S_{i \rightsquigarrow j }|$, is its number of transitions $m$, and we define shortest possible paths from $i$ to $j$ accordingly (note that there might be multiple paths with the same length between two states). 
  
  Denote the set of shortest possible paths $i \rightsquigarrow j$ by $\mathbf{\hat{S}}_{i \rightsquigarrow j}$, and let $N_{ij}$ be the length of the shortest possible paths from $i$ to $j$ (with the convention $N_{ii}=0$), i.e. $N_{ij} = |S_{i \rightsquigarrow j}|$ for all $S_{i \rightsquigarrow j} \in \mathbf{\hat{S}}_{i \rightsquigarrow j}$. We observe that all paths in $\mathbf{\hat{S}}_{i \rightsquigarrow j}$ have no self-transitions. We define $\hat{N} := \max_{i,j} N_{ij}$. Trivially, $ \hat{N} \le n+1$.
	\end{defn}
	
	The analogue of a primitive generator (for discrete time chains) is irreducibility of the  matrix $Q$: we say that $Q$ is irreducible if it is not permutation-similar to a block-upper-triangular matrix, or equivalently, if for all $i,j$, there exists a possible path $i\rightsquigarrow j$. This is the minimal condition under which our Markov chain is ergodic, or equivalently, such that $\exp(Qt)$ has all entries positive for all $t>0$, and should therefore have a Birkhoff coefficient less than $1$.
	
	\begin{prop}\label{prop:Qpower}
	   Let $Q$ be an irreducible $Q$-matrix. Then for all $i,j \in E$, we have $(Q^m)_{ij}=0$ for $m<N_{ij}$, and 
		\[(Q^{N_{ij}})_{ij} \geq \kappa_{ij}^{N_{ij}} \geq \kappa^{N_{ij}} >0, \]
		where  $\kappa_{ij} := \min_{S_{i \rightsquigarrow j} \in \mathbf{\hat{S}}_{i \rightsquigarrow j} } \{ Q_{kl} \, : \, k, l \in S_{i \rightsquigarrow j} \}$ and $\kappa := \min_{k \neq l } \{ Q_{kl} \, : \, Q_{kl} > 0\}$ .
	\end{prop}
	\begin{proof}
		As $Q$ is irreducible, there exist possible paths between any pair of states. As $Q$ is a $Q$-matrix, any possible path $S_{i\rightsquigarrow j} = (s_k)_{k=0}^m$ of length $m$ without self-transitions satisfies $\prod_{a=1}^m Q_{s_{a-1}s_a} >0$, with $s_0 = i$ and $s_m = j$.  Expanding the $m^{\mathrm{th}}$ power of $Q$, 
		\[(Q^m)_{ij} = \sum_{s_1,...,s_{m-1}} Q_{i s_1}Q_{s_1 s_2}\cdots Q_{s_{m-1} j} = \sum_{s_1,...,s_{m-1}} Q_{i s_1} \Big(\prod_{a=2}^{m-1} Q_{s_{a-1}s_a}\Big) Q_{s_{m-1} j}.\]
		By definition, all paths of length $m<{N_{ij}}$ are not possible, so $(Q^m)_{ij}=0$ for $m<{N_{ij}}$. Further, any path $i\rightsquigarrow j$ of length ${N_{ij}}$ is either not possible (so does not contribute to the sum), or is a shortest path, which implies it belongs to $\mathbf{\hat{S}}_{i \rightsquigarrow j}$, and in particular has no self-transitions. It follows that $(Q^{N_{ij}})_{ij}$ is the sum of non-negative terms of the form $Q_{i s_1} \Big(\prod_{a=2}^{N_{ij}-1} Q_{s_{a-1}s_a}\Big) Q_{s_{N_{ij}-1} j}$, each associated to a path in $\mathbf{\hat{S}}_{i \rightsquigarrow j}$. Minimizing over all possible shortest paths (i.e. over $\mathbf{\hat{S}}_{i \rightsquigarrow j}$), we have that the smallest term in the sum must be at least of size $\kappa^{N_{ij}}>0$.
	\end{proof}
	
	\begin{lemma}\label{lemma:Qpower2}
		Let $Q$ be a $Q$-matrix on $n+1$ states. Then $|(Q^m)_{ij}|\leq K^m $ for all $i,j \in E$, where $K=\|Q\|_{\mathrm{op}}=\sup_{\|a\|\leq 1}\|Qa\|= 2\max_i|Q_{ii}|$, with $\|\cdot\|$ the $\ell_\infty$ norm on $\mathbb{R}^{n+1}$.
	\end{lemma}
	\begin{proof}
		By a standard recursion, $\| Q^m\|_{\mathrm{op}}\leq K^m$. The explicit bound comes from observing $\|Q\|_{\mathrm{op}} = \max_{0 \le i \le n} \sum_j|Q_{ij}| = 2\max_i|Q_{ii}|$.
	\end{proof}
	\begin{rmk}
		Alternative norms can be used to give tighter bounds on the value $K$ (which is only used to bound $|(Q^m)_{ij}|$ in what follows), if the Markov chain is known to have particular structure.
	\end{rmk}
	
	\begin{lemma}\label{lemma:Qexpbounds}
		Let $Q$ be an irreducible $Q$-matrix on $n+1$ states. Write  $P(t) = \exp(Qt)$ and
		$K = \|Q\|_{\mathrm{op}}$. Then for any $\varepsilon\geq 0$,
		\[\Big|P(\varepsilon)_{ij} - \frac{1}{N_{ij}!}(Q^{N_{ij}})_{ij}\varepsilon^{N_{ij}}\Big| \leq  \frac{(K\varepsilon)^{N_{ij}+1}}{(N_{ij}+1)!}e^{K\varepsilon}, \quad \forall i,j \in E.\]
	\end{lemma}
	\begin{proof}
		As $(Q^{m})_{ij}=0$ for all $m<N_{ij}$,  we have
		\[P(\varepsilon)_{ij} = \frac{1}{N_{ij}!}(Q^{N_{ij}})_{ij}\varepsilon^{N_{ij}}+ \sum_{m>N_{ij}} \frac{1}{m!}(Q^{m})_{ij}\varepsilon^m,\]
		and hence, as $(m+N_{ij}+1)!\geq m!(N_{ij}+1)!$,
		\begin{align*}
			\Big|P(\varepsilon)_{ij} - \frac{1}{N_{ij}!}(Q^{N_{ij}})_{ij}\varepsilon^{N_{ij}}\Big|&\leq \frac{\varepsilon^{N_{ij}+1}}{(N_{ij}+1)!}\sum_{m\geq 0} \frac{1}{m!}|(Q^{m+N_{ij}+1})_{ij}|\varepsilon^m\\
			&\leq \frac{(K\varepsilon)^{N_{ij}+1}}{(N_{ij}+1)!}\sum_{m\geq 0} \frac{1}{m!}(K\varepsilon)^{m}\\
			&= \frac{(K\varepsilon)^{N_{ij}+1}}{(N_{ij}+1)!}e^{K\varepsilon}.
		\end{align*}
	\end{proof}

	\begin{thm}\label{thm:Birkhoff_irreducible_cts}
		Let $(X_t)_{t \ge 0}$ be a continuous-time Markov chain on $E \approxeq \{0, \dots n\}$ with irreducible $Q$-matrix $Q = (Q_{ij}) \in \R^{(n+1) \times (n+1)}$. Write $K=\|Q\|_{\mathrm{op}}= 2\max_{i}|Q_{ii}|$ and $\kappa = \min_{\{kl:Q_{kl}>0\}}Q_{kl}$, and $N_{ij}$ for the length of the shortest possible paths $i\rightsquigarrow j$. Let $\tau(P(t)) = \sup_{\mu,\nu}\mathcal{T} (\mu^\intercal P(t), \nu^\intercal P(t))$ denote the Birkhoff coefficient associated with the transition matrix $P(t) = \exp(Qt)$.
		
		For any $\alpha>1$, $\beta\in (0,1)$, for all $\varepsilon\leq \bar\varepsilon$, we know 
		\[\tau(P(\varepsilon)) \leq \frac{1-\eta\varepsilon^{\rho}}{1+\eta\varepsilon^{\rho}},\]
		where 
		\begin{align*}
			\rho & = \frac{1}{2}\max_{ij}\{N_{ij}+N_{ji}\}\leq \hat N, &\hat N &=\max_{ij}N_{ij},\\
			\bar\varepsilon&= 1\wedge \frac{\log(\alpha)}{K} \wedge  \frac{\beta}{\alpha }(\hat N+1)\Big(\frac{\kappa}{K}\Big)^{\hat N+1}, &
			\eta &=\frac{1-\beta}{1+\beta}\Big(\frac{\kappa}{K}\Big)^{\hat N}\frac{1}{\hat N!}.
		\end{align*}
	\end{thm}
	\begin{corollary}\label{cor:pasted_tau_estimate}
		With all parameters as above, for all $t>0$,  by separating $[0,t]$ into intervals of length at most $\varepsilon\leq \bar\varepsilon$, the sub-multiplicativity of $\tau$ implies
		\[\tau(P(t)) \leq \Big(\frac{1-\eta\varepsilon^{\rho}}{1+\eta\varepsilon^{\rho}}\Big)^{\lfloor t/\varepsilon\rfloor}\Big(\frac{1-\eta(t-\bar\varepsilon\lfloor t/\varepsilon\rfloor)^{\rho}}{1+\eta(t-\bar\varepsilon\lfloor t/\varepsilon\rfloor)^{\rho}}\Big) \leq  \Big(\frac{1-\eta\varepsilon^{\rho}}{1+\eta\varepsilon^{\rho}}\Big)^{t/\varepsilon-1}.\]
	\end{corollary}
	\begin{proof}[Proof of Theorem \ref{thm:Birkhoff_irreducible_cts}]
		We consider the quantity
		\[\varphi_{ijkl}(\varepsilon) := \frac{P(\varepsilon)_{ik}P(\varepsilon)_{jl}}{P(\varepsilon)_{il}P(\varepsilon)_{jk}}.\]
		For ease of notation, we write 
		\[A_{ij} = \frac{1}{N_{ij}!}(Q^{N_{ij}})_{ij}, \qquad B_{ij} = \frac{\alpha K^{N_{ij}+1}}{(N_{ij}+1)!},\]
		so that $A_{ij}>0$ and $B_{ij}>0$, and Lemma \ref{lemma:Qexpbounds} guarantees
		\begin{align*}
			\Big|P(\varepsilon)_{ij} - A_{ij}\varepsilon^{N_{ij}}\Big| &\leq   \frac{(K\varepsilon)^{N_{ij}+1}}{(N_{ij}+1)!}e^{K\varepsilon} = B_{ij} \varepsilon^{N_{ij}+1} \frac{e^{K\varepsilon}}{\alpha}.
		\end{align*}
		As $\varepsilon\leq \frac{\log(\alpha)}{K}$, we  know $\frac{e^{K\varepsilon}}{\alpha}\in[0,1]$, and hence
  \[A_{ij}\varepsilon^{N_{ij}} - B_{ij}\varepsilon^{N_{ij}+1}\leq P(\varepsilon)_{ij} \leq A_{ij} \varepsilon^{N_{ij}} + B_{ij}\varepsilon^{N_{ij}+1}.\]
  Substituting in the definition of $\varphi_{ijkl}$, we have
		\begin{equation}\label{eq:phibound}
			\begin{split}
				\varphi_{ijkl}(\varepsilon)& \geq  \frac{\Big(A_{ik}\varepsilon^{N_{ik}}-B_{ik}\varepsilon^{N_{ik}+1}\Big)\Big(A_{jl}\varepsilon^{N_{jl}}-B_{jl}\varepsilon^{N_{jl}+1}\Big)}{\Big(A_{il}\varepsilon^{N_{il}}+B_{il}\varepsilon^{N_{il}+1}\Big)\Big(A_{jk}\varepsilon^{N_{jk}}+B_{jk}(\varepsilon)^{N_{jk}+1}\Big)}\\
				&=\frac{\Big(A_{ik}-B_{ik}\varepsilon\Big)\Big(A_{jl}-B_{jl}\varepsilon\Big)}{\Big(A_{il}+B_{il}\varepsilon\Big)\Big(A_{jk}+B_{jk}\varepsilon\Big)}\varepsilon^{N_{ik}+N_{jl}- N_{il}-N_{jk}}.
		\end{split}\end{equation}
		By Proposition \ref{prop:Qpower} we know that $(Q^{N_{ij}})_{ij}\geq \kappa^{N_{ij}}$ for all $i,j$. Therefore, 
		\[
		\min_{ij}\Big\{\frac{A_{ij}}{B_{ij}}\Big\} = \min_{ij}\Big\{\frac{(N_{ij}+1)(Q^{N_{ij}})_{ij}}{\alpha K^{N_{ij}+1}}\Big\}\geq\min_{ij}\Big\{\frac{(N_{ij}+1)}{\alpha (K/\kappa)^{N_{ij}+1}}\Big\} = \frac{1}{\alpha}(\hat N+1) \Big(\frac{\kappa}{K}\Big)^{\hat N+1}
		\]
		where the final equality follows from observing that the map $x\mapsto xy^{-x}$ is decreasing in $x$ for $x\ge-\log(y)^{-1}$, and we know $K/\kappa\geq 2$ and $N_{ij}+1\geq 2$. From the definition of $\bar\varepsilon$,
		\[\bar\varepsilon\leq  \frac{\beta}{\alpha}(\hat N+1) \Big(\frac{\kappa}{K}\Big)^{\hat N+1} \leq \beta \min_{ij}\Big\{\frac{A_{ij}}{B_{ij}}\Big\},\]
		and hence, for $\varepsilon\leq\bar\varepsilon$, we know that  $A_{ij}-B_{ij}\varepsilon \geq (1-\beta)A_{ij}$ and $A_{ij}+B_{ij}\varepsilon\leq (1+\beta)A_{ij}$ for all $i,j$. Returning to \eqref{eq:phibound},
		\[      \varphi_{ijkl}(\varepsilon) \geq \Big(\frac{1-\beta}{1+\beta}\Big)^2
		\frac{A_{ik}A_{jl}}{A_{il}A_{jk}}\varepsilon^{N_{ik}+N_{jl}- N_{il}-N_{jk}},\]
		and hence
		\begin{align*}
			\frac{1}{2}\log\varphi_{ijkl}(\varepsilon) &\geq 
			\frac{N_{ik}+N_{jl}- N_{il}-N_{jk}}{2}\log(\varepsilon) +\frac{1}{2}\log\frac{A_{ik}A_{jl}}{A_{il}A_{jk}}+\log\frac{1-\beta}{1+\beta}.
		\end{align*}
		As the $N_{ij}$ are the lengths of shortest paths, we know that $N_{ik}-N_{il}\leq N_{lk}$ for all $i,k,l$. Therefore,
		\[0\leq  N_{ik}+N_{jl}- N_{il}-N_{jk} \leq N_{lk} + N_{kl}\]
		and the upper bound is attained if $i=l\neq j=k$. As $\varepsilon\leq \bar\varepsilon\leq 1$, this implies 
		\begin{equation}\label{eq:phibound2}
			\frac{1}{2}\log\varphi_{ijkl}(\varepsilon) 
			\geq 
			\rho \log(\varepsilon)+\frac{1}{2}\log\frac{A_{ik}A_{jl}}{A_{il}A_{jk}}+\log\frac{1-\beta}{1+\beta}.
		\end{equation}
		Using Proposition \ref{prop:Qpower} and Lemma \ref{lemma:Qpower2},
		\begin{equation}\label{eq:Abound}
			\frac{1}{2}\log\frac{A_{ik}A_{jl}}{A_{il}A_{jk}} =  \frac{1}{2}\log\frac{(Q^{N_{ik}})_{ik}(Q^{N_{jl}})_{jl}}{(Q^{N_{il}})_{il}(Q^{N_{jk}})_{jk}}-\frac{1}{2}\log\frac{N_{ik}!N_{jl}!}{N_{il}!N_{jk}!}
			\geq  \hat N\log\frac{\kappa}{K}-\log(\hat N!)
		\end{equation}
		from which we see 
		\[   \frac{1}{2}\log\varphi_{ijkl}(\varepsilon) \geq \rho\log(\varepsilon)+\log \eta.\]
		
		We can now compute the Birkhoff coefficient, which we know from Proposition \ref{prop:contraction_coeff} can be written
		\[
		\tau(P(\varepsilon)) = \frac{1-\sqrt{\min_{ijkl}\varphi_{ijkl}(\varepsilon)}}{1+\sqrt{\min_{ijkl}\varphi_{ijkl}(\varepsilon)}}\leq \frac{1-\eta\varepsilon^{\rho}}{1+\eta\varepsilon^{\rho}}.
		\]
	\end{proof}
	
	\begin{rmk}
		The value of $\rho$ has an elegant interpretation as half the length of the shortest loop $i\rightsquigarrow j \rightsquigarrow i$, maximized over $i,j$. If $Q_{ij}>0 \Leftrightarrow Q_{ji}>0$, then this is simply the diameter $\hat N$ of the underlying (undirected) graph of the Markov chain.
		
		In the case where all transitions are possible, so $N_{ij} =1 $ for all $i\neq j$, we see that $\rho=1$. In this case, it is possible to recover a result similar to Proposition \ref{prop:MC_cts_time_base} using Corollary \ref{cor:pasted_tau_estimate} and the limits $\varepsilon\to 0$ then $\beta\to 0$. The rate we achieve is poorer, mainly due to the inequality in \eqref{eq:Abound}; this can be tightened by taking $\varepsilon$ sufficiently small that we only need consider $i=l\neq j=k$ when deriving \eqref{eq:phibound2}, in which case the right hand side of \eqref{eq:Abound} becomes $\frac{1}{2}\min_{ij}\{\log((Q^{N_{ij}})_{ij}(Q^{N_{ji}})_{ji})) - \log(N_{ij}!N_{ji}!)\}$. With this change, we recover the rate in Proposition \ref{prop:MC_cts_time_base}.
		
		For $\rho>1$, we see that there exists $\lambda\geq 2\eta$ such that $\tau(P(\varepsilon)) = 1 -\lambda \varepsilon^\rho +O(\varepsilon^{2\rho})$, and one can check that this is the correct asymptotic behaviour as $\varepsilon\to 0$. Heuristically, this suggests that obtaining a larger value of $\bar\varepsilon$ is more significant than a large value of $\eta$, when it comes to the resulting rate estimate.
	\end{rmk}
	
	\newpage
	
	%%% bibliography
	\bibliographystyle{plain}
	\bibliography{biblio}
	
\end{document}